\documentclass[11pt]{amsart}
\usepackage{amscd,amsmath,amssymb,amsfonts}
\usepackage{color}
\usepackage[cmtip, all]{xy}
\usepackage{graphicx}
\usepackage{enumerate}
\usepackage{appendix}
\numberwithin{equation}{section}
\newtheorem{thm}[equation]{Theorem}

\newtheorem{lemma}[subsection]{Lemma}

\newtheorem{proposition}[subsection]{Proposition}
\newtheorem{prop}[subsection]{Proposition}

\newtheorem{definition}[equation]{Definition}

\newtheorem{cor}[subsection]{Corollary}
\newtheorem{hyp}[equation]{Hypothesis}
\newtheorem{defn}[subsection]{Definition}
\newtheorem{conj}[subsection]{Conjecture}

\newtheorem*{thm*}{Theorem}
\theoremstyle{remark}

\newtheorem{remark}[subsection]{Remark}

\newcommand{\isoarrow}{{~\overset\sim\longrightarrow~}}
\newcommand{\lra}{{\longrightarrow}}
\newcommand{\bT}{{\mathbf{T}}}
\newcommand{\CO}{{\mathcal{O}}}
\newcommand{\CW}{{\mathcal{W}}}
\newcommand{\CE}{{\mathcal{E}}}
\newcommand{\CP}{{\mathcal{P}}}
\newcommand{\ord}{\mathrm{ord}}
\newcommand{\aord}{\mathrm{a-ord}}

\newcommand{\fp}{{\mathfrak{p}}}

\newcommand{\Ts}{{T^{sym}}}

\newcommand{\ZZ}{{\mathbb Z}}
\newcommand{\TT}{{\mathbb T}}
\newcommand{\Gm}{{\mathbb G}_m}
\newcommand{\GmC}{{\mathbb G}_{m,\CC}}

\newcommand{\CG}{{\mathcal{G}}}
\newcommand{\C}{{\mathcal{C}}}

\newcommand{\CK}{{\mathcal{K}}}

\newcommand{\CV}{{\mathcal{V}}}
\newcommand{\CL}{{\mathcal{L}}}

\newcommand{\CJ}{{\mathcal{J}}}

\newcommand{\DD}{{\mathbb D}}
\newcommand{\fm}{{\mathfrak{m}}}
\newcommand{\fg}{{\mathfrak{g}}}
\newcommand{\fk}{{\mathfrak{k}}}

\newcommand{\fn}{{\mathfrak{n}}}
\newcommand{\kp}{{\kappa'}}

\newcommand{\ra}{{~\rightarrow~}}

\newcommand{\CH}{{\mathcal{H}}}

\newcommand{\Ss}{{\mathbb S}}
\newcommand{\QQ}{{\mathbb Q}}
\newcommand{\LL}{{\mathbb L}}
\newcommand{\sS}{{\mathbb S}}

\newcommand{\Qbar}{{\overline{\mathbb Q}}}

\newcommand{\Qlb}{{\overline{\mathbb Q}_{\ell}}}

\newcommand{\Zp}{{{\mathbb Z}_p}}

\newcommand{\RR}{{\mathbb R}}
\newcommand{\ad}{{\mathbf A}}
\newcommand{\al}{{\alpha}}
\newcommand{\af}{{{\mathbf A}_f}}
\newcommand{\CC}{{\mathbb C}}

\newcommand{\Qp}{{{\mathbb Q}_p}}
\newcommand{\Qpb}{{\bar{\Qp}}} 

\begin{document}
\author{Michael Harris}

\address{Michael Harris\\
Department of Mathematics, Columbia University, New York, NY  10027, USA}
 \email{harris@math.columbia.edu}


\title[Square root $p$-adic $L$-functions, I:  Construction of a one-variable measure]{Square root $p$-adic $L$-functions, I:  Construction of a one-variable measure}
\dedicatory{For Jacques Tilouine}
\thanks{This work was partially supported by NSF Grant DMS-1701651.  This work was also supported by the National Science Foundation under Grant No. DMS-1440140 while the author was in residence at the Mathematical Sciences Research Institute in Berkeley, California, during the Spring 2019 semester.}
\maketitle


\tableofcontents


\section{Introduction}   This paper is a continuation, after twenty years (!), of the author's project \cite{HT} with Jacques Tilouine, whose official goal was the construction of one branch of the square root of the anticyclotomic $p$-adic $L$-function for a triple of classical modular forms.  
The unofficial goal of that paper was for this author to benefit from Jacques's patient instruction in Hida theory and $p$-adic $L$-functions.   To the extent that the author does understand anything about the subject, it is largely a result of this collaboration.

Our paper was neither the first nor the last word on the topic of square root $p$-adic $L$-functions.  
The bibliography of \cite{HT} included references to earlier work on anticyclotomic $L$-functions of Hecke characters of imaginary quadratic fields, and of classical $L$-functions of modular forms in Hida families, as well as a combination of the two that had been considered by Andrea Mori (finally published, more than 20 years after its discovery, in \cite{Mo}).  
The specific case of the triple product was vastly extended (and corrected) and put to good use by Darmon and Rotger in a series of difficult papers  on Euler systems and the Birch-Swinnerton-Dyer conjecture over nonabelian extensions of $\QQ$ (see \cite{DR}).  
More recently, the construction has been generalized to Shimura curves in \cite{BM}.  
In the meantime, Gan, Gross, and Prasad had identified a natural setting that includes all these special cases \cite{GGP}, and had formulated precise conjectures regarding the relative representation theory of certain pairs of reductive groups over local fields.  
These conjectures were completed by the global conjecture of Ichino-Ikeda (for orthogonal groups) and its analogue, due to N. Harris (for unitary groups) \cite{II,NH}.  
In these conjectures, $G \supset H$ is a pair of groups -- we consider the case where $G$ is the special orthogonal or unitary group of a vector space $V$ over a local or global field and $H$  the stabilizer in $G$ of an appropriate subspace of codimension $1$.  
The conjectures of \cite{GGP} classify the irreducible representations $\pi$ of $G \times H$ over a {\it local field} that admit a linear form $\pi \ra \CC$ that is invariant under $H$, with respect to the diagonal embedding.  
The conjectures of \cite{II,NH} concern the cuspidal automorphic representations $\pi$ of $G \times H$ over a {\it number field}, and express  the central (anticyclotomic) values of certain  $L$-functions $L(s,\pi)$ as squares of periods of integrals over the ad\`ele group of $H$ of elements of $\pi$ -- we call them {\it Gan-Gross-Prasad periods} up to local and elementary factors.

The Gan-Gross-Prasad (GGP) conjectures have been proved by Waldspurger (for orthogonal groups over $p$-adic fields) and Beuzart-Plessis (for unitary groups, including the archimedean case) \cite{W,BP}.  
The Ichino-Ikeda-N. Harris (IINH) conjecture has been proved for unitary groups, under a shrinking set of simplifying hypotheses, by W. Zhang,  followed by Hang Xue and Beuzart-Plessis; at present it is known for any everywhere tempered automorphic representation of a unitary group that is either supercuspidal at one place or is stable, in the sense that its base change to $GL(n)$ (see below) is cuspidal.   
Ichino had already proved the conjecture for orthogonal groups in low dimensions, including a refinement of the result of \cite{HK} on triple products that was the starting point for \cite{HT}.  
The main observation of \cite{HT} is that the period integrals in \cite{HK} admit a $p$-adic interpolation over Hida families.  
The purpose of the present paper is to apply the same observation to the period integrals that arise in certain cases of the IINH conjecture when $G$ and $H$ are unitary groups.  

Suppose $G$ and $H$ are the unitary groups of hermitian vector spaces $V$ and $V'$, respectively, over a fixed imaginary quadratic field $\CK$,\footnote{One can consider more general CM quadratic extensions of totally real fields; the restriction to the imaginary quadratic case is made for convenience.} with $\dim V = n = \dim V' + 1$.  Stable quadratic base change from $G \times H$ to $\CG := GL(n)_\CK \times GL(n-1)_\CK$
(\cite{La11, Mok, KMSW}) attaches to a (stable) $\pi$ a cuspidal automorphic representation $\Pi = \Pi_n \boxtimes \Pi_{n-1}$ of $\CG$, and $L(s,\pi)$ is then the Rankin-Selberg $L$-function $L(s,\Pi_n \otimes \Pi_{n-1})$.   Moreover, as $G_\alpha \times H_\alpha$
varies over inner forms of $G \times H$, with $H_\alpha \subset G_\alpha$, there is a collection $\Phi(\Pi) = \{\pi_{\alpha,\beta}\}$, where
each  $\pi_{\alpha,\beta}$ is cuspidal automorphic representation of  $G_\alpha \times H_\alpha$, all of which have the same stable
base change $\Pi$.  We only consider the case where $\Pi$ is cohomological for $\CG$ -- i.e., it contributes to the cuspidal cohomology of the locally symmetric space for $\CG$ with appropriate local coefficients.  Then the archimedean components $\pi_{\alpha,\beta,\infty}$ all belong to the respective discrete series of $G_\alpha(\RR) \times H_\alpha(\RR)$ whose infinitesimal character corresponds to that of $\Pi_\infty$.  
The combination of the GGP and IINH conjectures includes the assertion that the central value $L(\frac{1}{2},\Pi_n \otimes \Pi_{n-1}) \neq 0$, then it (or rather its ratio to a different special $L$-value) can be computed, up to elementary and local factors, as a ratio of a product of period integrals (over the adeles of $H_\alpha$) of a {\it unique} $\pi_{\alpha,\beta} \in \Phi(\Pi)$.  Specifically, the GGP conjecture, which is known in this case, asserts that is a unique group $G_\alpha(\RR) \times H_\alpha(\RR)$ and a unique discrete series $\pi_{\alpha,\beta,\infty}$ with the given infinitesimal character that admits a non-trivial linear form
$$\pi_{\alpha,\beta,\infty} \ra \CC$$
that is invariant under the diagonal embedding of $H_\alpha(\RR)$.

As $\Pi$ varies in a $p$-adic family, the period integrals for the corresponding $\pi_{\alpha,\beta,\infty}$ can be seen as distinct branches of a hypothetical square root $p$-adic $L$-function, the relations between which have only begun to be explored; the example of \cite{DR} shows that these relations are subtle even in low-dimensional cases.   
In this paper we treat the branch where the specialization of the $p$-adic interpolation at a classical point of the Hida family is a cup product of a pair of automorphic forms, one of which is holomorphic, the other anti-holomorphic -- in other words, 
$\pi_{\alpha,\beta,\infty} = \pi_{n,\infty}\otimes \pi_{n-1,\infty}$, where $\pi_{n,\infty}$ (resp. $\pi_{n-1,\infty}$) is a holomorphic 
(resp. anti-holomorphic) discrete series representation of $U(V)(\RR)$ (resp. $U(V')(\RR)$).    
Here we treat only the simplest case of a function of a single $p$-adic variable, which arises as a direct application of the construction of $p$-adic families of differential operators in \cite{EFMV}.  
A planned sequel with Ellen Eischen should extend the results of the present paper to multidimensional Hida families.  
In subsequent work with Eischen and Pilloni, we hope to treat cases of cup products in coherent cohomology of higher degree.  Most of the GGP periods, however, do not have such an interpretation.  
The corresponding $p$-adic $L$-functions should exist nonetheless, but we don't see how to construct them.

A construction of $p$-adic Rankin-Selberg $L$-functions for cohomological automorphic representations $\Pi$ 
of $GL(n)\times GL(n-1)$ over $\QQ$ has been known for some time \cite{KMS}; its generalization to arbitrary number fields is a more recent results of Januszewski \cite{J}.  The method used there corresponds to the ``branch,'' as in the previous paragraph, where the groups $G_\alpha(\RR)$
and $H_\alpha(\RR)$ are {\it definite} unitary groups.   This is precisely the case in which the  methods of the present paper 
give no $p$-adic variation at all, because there are no non-trivial differential operators.  In unpublished notes, Eric Urban has sketched the beginning of a construction of a $p$-adic measure in this situation, again in the definite branch.   
In any case, there is little overlap between the results of \cite{KMS,J}, which treat general critical values of a single Rankin-Selberg $L$-function and its cyclotomic twists, and those of this paper, which treats only the central value but allows $\Pi$ to vary in a $p$-adic family.  

The hardest steps in the construction of any $p$-adic $L$-function are the computation of the local factors at archimedean and $p$-adic places.  We deal with these steps in the present paper by avoiding them.  The Ichino-Ikeda formula produces local factors at such places and we do not attempt to interpret them explicitly.   It follows nevertheless from \cite{BP} that these factors can be computed in terms of local Rankin-Selberg zeta integrals for $GL(n)\times GL(n-1)$.  These should be easier to
compute than the Ichino-Ikeda local integrals.  We expect to return to these computations in subsequent papers.  

I thank Ellen Eischen for help with the $p$-adic differential operators, and Rapha\"el Beuzart-Plessis for discussion of the points mentioned in \S\S \ref{locp}, \ref{locinf}.  I thank Eric Urban for reminding me that (in unpublished work) he had considered a  $p$-adic interpolation of periods  in the definite case -- the method is roughly orthogonal to the one studied here.   The anonymous referee deserves special thanks for a careful reading and for requesting clarification at a number of points; it was only by addressing the referee's questions that I realized the subtle difference between the contraction in \cite{EHLS} and the version constructed in \S \ref{contraction}.  Most of all, I would like to take this opportunity to thank Jacques Tilouine for teaching me most of what I know about the subject, and for his  friendship over many decades.

\section{Unitary group Shimura varieties}\label{sec2}

We work over an imaginary quadratic field $\CK$; most of our results go over without change to general CM fields, at the cost of more elaborate notation.  The field $\CK$ is given with a chosen embedding $\iota:  \CK \hookrightarrow \CC$; the complex conjugate embedding is denoted $c$; with respect to $\iota$, the group 
$$U(1) = \ker N_{\CK/\QQ}:  R_{\CK/\QQ} GL(1) \ra GL(1)$$
can be attached to a Shimura datum $(U(1),Y_1)$, where $Y_1^{\pm}$ is the homomorphism  
$$R_{\CC/\RR}(\Gm)_\CC = \CC^\times \ni z \mapsto z/\bar{z}$$
if the sign is $-1$ and is the {\it trivial} map if the sign is $+1$.
The sign is $+1$ (resp. $-1$)
if we consider $U(1)$ to be the unitary group of a $1$-dimensional vector space over $\CK$ endowed with a hermitian form of signature $(0,1)$ (resp. $(1,0)$); see
the discussion in \cite{H19}, \S 2.2 for details.

Let $V$ be an $n$-dimensional vector space over $\CK$, endowed with a hermitian form of signature $(r,s)$ (relative to $\iota$).  Let $U(V)$ be the unitary group of $V$. We define a Shimura datum $(U(V),Y_V)$ as in \cite{H19}; see also Appendix \ref{ShU}.  We choose a point $y \in Y_V$ corresponding to an embedding of Shimura data $(U(1),Y_1^{\pm}) \hookrightarrow (U(V),Y_V)$, and let $K_y \subset U(V)(\RR)$ denote its centralizer; in other words, the homomorphism $y$ factors through a rational subgroup of $U(V)$ isomorphic to $U(1)$.   Then there is an isomorphism $K_y \isoarrow U(r) \times U(s)$, where $U(d)$ is the compact unitary group of rank $d$ for any $d$.  We fix a maximal torus $T = T_y \subset K_y$ containing the chosen $U(1)$; without loss of generality we may assume $T \isoarrow U(1)^{r+s}$ as algebraic groups over $\QQ$, with $U(1)^r \subset U(r)$ and $U(1)^s \subset U(s)$.  

The Harish-Chandra decomposition of $\fg = Lie(G_V)$ is given by
$$\fg = \fp^+_y \oplus \fp^-_y \oplus \fk_y$$
where $\fk_y = Lie(K_y)$ and $\fp^+_y$ and $\fp^-_y$ are canonically isomorphic, respectively, to the holomorphic and anti-holomorphic tangent spaces to $Y_V$ at $y$.  Then $\dim \fp^+_y = \dim \fp^-_y = rs$.  
 
\subsection{Conventions for holomorphic automorphic forms}\label{holomo}
Irreducible representations of $U(d)$ are parametrized by $d$-tuples $a_1 \geq a_2 \geq \dots \geq a_d$, which are identified with characters of some chosen maximal torus.  Thus irreducible automorphic vector bundles $\CE_\kappa$ over the Shimura variety $Sh_V := Sh(U(V),Y_V)$ are parametrized by characters of $T_y$, and thus of $(r,s)$-tuples of integers 
\begin{equation}\label{rstuple}  (b_1 \geq b_2 \geq \dots \geq b_r; b_{r+1}  \geq b_{r+2} \geq \dots \geq b_n).
\end{equation} 
The $\CE_\kappa$ whose global sections defined holomorphic automorphic forms in the discrete series correspond to $\kappa$ of the form
\begin{equation}\label{DS}
\kappa =(a_{s+1} - s, \dots  a_n -s; a_1 + r, \dots, a_s + r)
\end{equation}
where $\alpha$ is the dominant parameter
\begin{equation}\label{alpha}  \alpha: a_1 \geq a_2 \geq \dots \geq a_n
\end{equation} 
\cite[Proposition 2.2.7 (iii)]{H97}.  A $\kappa$ satisfying \eqref{DS} will be called {\it of holomorphic type}.  We let $M_\kappa$ denote the representation space of $K_y$ with highest weight $\kappa$, and let
$$\DD_\kappa = U(\fg)\otimes_{U(\fk_y \oplus \fp^-)} M_\kappa$$
be the corresponding holomorphic discrete series.  

A $\kappa$ for which $\DD_\kappa$ is the $(\fg,K_y)$-module attached to a discrete series representation will be called (for convenience) a {\it holomorphic discrete series parameter}.  If $\kappa$ is the parameter of \eqref{DS}, it is determined by \eqref{alpha}, and we write $\kappa = \kappa_V(\alpha)$.

We consider a codimension $1$ hermitian subspace $V' \subset V$, of signature $(r,s-1)$, and we assume that the base point $y \in Y_{V'} \subset Y_V$, so that its centralizer $K'_y \subset U(V')(\RR)$ is a maximal compact subgroup, isomorphic to $U(r)\times U(s-1)$.  We write
$$\fg' = \fp^{+,\prime}_y \oplus \fp^{-,\prime}_y \oplus \fk'_y$$
for the Harish-Chandra decomposition of $\fg' = Lie(U(V'))$.  
As representation of $K'_y = U(r)\times U(s-1)$, the $r$-dimensional quotient space
$$\fn = \fp^+_y/\fp^{+,\prime}_y$$ is isomorphic to the representation $St_r \otimes Triv$, with parameter $(1,0,\dots,0;0,\dots,0)$.  It follows from the recipe in \cite{H86} that the restriction of $\DD_\kappa$ to $U(\fg')$ can be written
\begin{equation}\label{reskappa}
\DD_{\kappa} ~|_{U(\fg')} = \bigoplus_{i \geq 0} \bigoplus_{M_{\kappa'} \subset M_\kappa\otimes [Sym^iSt_r \otimes Triv]}  \DD_{\kappa'}.
\end{equation}
Here the notation $\subset$ in the subscript means that the left-hand representation is an irreducible constituent of the restriction to $U(r)\times U(s-1)$ of the right-hand.  

In what follows, we let $T' = T_y \cap U(V') = T_y \cap K'_y$.  This is a maximal CM torus in $U(V')$ and the parameters in Lemma \ref{param} below are relative to this torus.  The inclusion $(U(V'),Y_{V'}) \subset (U(V),Y_V)$ is not an embedding of Shimura data, but this can be corrected by replacing $U(V')$ by $U(V')\times U(1)$, where $U(1)$ is the unitary group of the orthogonal complement to $V'$ in $V$.  We ignore this for the purposes of this paper.

  \begin{lemma}\label{param}  If $\kappa = (a_{s+1} - s, \dots  a_n -s; a_1 + r, \dots, a_s + r)$, then as  $i \geq 0$ varies, the set of irreducible representations of $K'_y$ contained in the above sum is given by :
\begin{equation*}
(b_1,\dots, b_r; c_1, \dots, c_{s-1}); 
\end{equation*}
where 
\begin{equation*}  \delta_j := b_j  + s - a_{s+j} \geq 0, ~~ 1 \leq j \leq r;  
\end{equation*}
\begin{equation*} a_1 + r \geq c_1 \geq a_2+r \geq c_2 \dots \geq c_{s-1} \geq a_s + r.
\end{equation*}
The parameter arises in degree $i$ exactly when $\sum_{j = 1}^r \delta_j = i$. 
\end{lemma}
\begin{proof}  The assertion for the $b_j$'s follows from the Littlewood-Richardson rule \cite[\S 9.3]{GW} for the tensor product of an irreducible representation of $U(r)$ with $Sym^i St_r$, given our sign conventions; the assertion for the $c_k$'s follows from the usual branching formula for restriction from $U(s)$ to $U(s-1)$.  
\end{proof}

It thus follows from \cite[Lemma 7.2]{H86} that
\begin{prop}\label{diffops}  Let $\kappa$ be a holomorphic discrete series parameter.  Let $\kappa'$ be the highest weight of an irreducible representation of $K'_y$.  Then there is a holomorphic differential operator
$$\delta_{\kappa,\kappa'}:  \CE_\kappa ~|_{Sh(V')} \ra \CE_{\kappa'}$$
if and only if  $\kappa'$ satisfies the inequalities of Lemma \ref{param}.  
\end{prop}

The following lemma is then obvious.

\begin{lemma}\label{GGP1}  Suppose $\kappa'$ satisfies the inequalities of Lemma \ref{param}.  Then $\kappa'$ is a holomorphic discrete series parameter for $G_{V'}$, and is of the form
$\kappa' = \kappa_{V'}(\alpha')$ for the dominant parameter $\alpha'$ of $G_{V'}$ given by
$$\alpha' = (a'_1 \geq \dots \geq a'_{n-1}) = (c_1-r \geq \dots \geq c_{s-1}-r \geq b_1 + s-1 \geq \dots \geq b_r + s-1). $$
\end{lemma}
\begin{definition}\label{deg} We say $\delta^{\kappa,\kappa'}$ is of degree $b$ if $\sum_{j = 1}^r \delta_j = b$ in Lemma \ref{param}.
\end{definition}

\subsubsection{Parameters and Hodge structures}

Let $\alpha$ be the dominant parameter in \eqref{alpha}.  Then $\alpha$ is the highest weight of an irreducible representation $W_\al$ of $G_V$, or of $GL(n)$.  As in \cite{H97} we can attach to $\alpha$ a collection of Hodge numbers $(p_i,q_i) = (p_i(\al),q_i(\al))$ with $p_i = a_i + n-i$ and $p_i + q_i = n-1$ for all $i$.  We let $\sS = R_{\CC/\RR} \GmC$.  For each $i$, let  $M_\CC(p_i)$ denote the complex $1$-dimensional vector space on which $\sS(\CC) \isoarrow \CC^\times \times \CC^\times$ acts by the character 
$$(z_1,z_2) \mapsto z_1^{-p_i(\al)}z_2^{-n+1+p_i(\al)},$$
and let $M_\CC(\al) = \oplus_{i = 1}^n M_\CC(p_i)$.  Similarly, let $M(p_i)$ denote $R_{\CC/\RR}M_\CC(p_i)$; this is a $2$-dimensional vector space with action of $\sS(\RR)$.  Then $M(\al) = \oplus_{i = 1}^n M(p_i)$ is a real Hodge structure of dimension $2n$.  We denote $M(\al)$ by the shorthand list of $p_i$'s:
\begin{equation}\label{Mal}
M(\al) = (a_1 + n-1, \dots, a_i + n-i,\dots, a_n).
\end{equation}

Let $\pi$ be a cuspidal automorphic representation of $G_V$, and write $\pi = \pi_\infty\otimes \pi_f$, where $\pi_\infty$ is an irreducible $(\fg_V,K_y)$-module and $\pi_f$ is an irreducible representation of $G_V(\af)$.  
Suppose $\pi$ contributes to the cohomology $H^0(Sh(V),\CE_{\kappa_V(\al)})$; in other words
\begin{equation}\label{pif} H^0(Sh(V),\CE_{\kappa_V(\al)})[\pi] := Hom_{G_V(\af)}(\pi_f,H^0(Sh(V),\CE_{\kappa_V(\al)}) \neq 0.
\end{equation}
This is a property that depends only on $\pi_\infty$; it says precisely that $\pi_\infty$ is (depending on conventions) either isomorphic to or the contragredient of  $\DD_{\kappa_V(\alpha)}$.  In the convention of \cite{EHLS}, 
\begin{hyp}\label{multone}  Assuming \eqref{pif}, $\dim H^0(Sh(V),\CE_{\kappa_V(\al)})[\pi] = 1$.   
\end{hyp} 
This will be proved in the sequel to \cite{KMSW}, and we will assume it here; the $p$-adic $L$-function can be constructed without 
the assumption of Hypothesis \ref{multone} but at the cost of additional notation.
In any case it is known by \cite{La11} that, assuming \eqref{pif}, the base change $\Pi = BC_{\CK/\QQ}(\pi)$ exists as a cuspidal cohomological automorphic representation of $GL(n)_\CK$.  The compatible family  of homomorphisms
$$\rho_{\pi,\ell}:  Gal(\Qbar/\CK) \ra GL(n,\Qlb),$$
defined by many people (including in \cite{CHL11}, in most cases) is geometric in the sense of Fontaine-Mazur.  In particular, the restriction of $\rho_{\pi,\ell}$ to a decomposition group at a prime dividing $\ell$ is de Rham with the Hodge numbers $(p_i(\al),q_i(\al))$ defined above.  

\begin{remark}\label{rationality}  Under hypothesis \eqref{pif} we know that $\pi_f$ has a model as an admissible representation of $G_V(\af)$ over a number field $E(\pi)$.  We will be working with spaces of $p$-adic automorphic forms, so we will implicitly be assuming that the integer ring $\CO_{E(\pi)}$, together with all the other integer rings that arise in the subsequent constructions, is embedded in a  sufficiently large $p$-adic integer ring denoted $\CO$.  We will briefly need to work with models of (finite parts of) automorphic representations over the fraction field of $\CO$, which we denote $\LL$.  The smooth representation theory of the finite adeles is  indifferent to the topology of the fields of coefficients.
\end{remark}

Suppose $\pi' = \pi'_\infty \otimes \pi'_f$ is an automorphic representation such that the contragredient $\pi^{\prime,\vee}$ contributes to the cohomology $H^0(Sh(V'),\CE_{\kappa_V(\al')})$, where $\alpha'$ is the highest weight of an irreducible representation $W_{\al'}$ of $GL(n-1)$.  In particular, $\pi'$ is {\it anti-holomorphic} -- it contributes to the cohomology in degree $d' = \dim Sh(V')$ of the automorphic vector bundle
$$\Omega^{d'}_{Sh(V')}\otimes \CE_{\kappa_V(\al')}^{\vee},$$
which is the Serre dual of $\CE_{\kappa_V(\al')}$.  For such a $\pi'$, we assume the analogue of Hypothesis \ref{multone} holds for $\pi^{\prime,\vee}$, and we assume $\kappa' = \kappa_{V(\al')}$ satisfies the inequalities of Lemma \ref{param}; in other words, that $\alpha'$ is one of the parameters in Lemma \ref{GGP1}.   

\subsubsection{Parameters for the Hodge filtration}\label{forPan}

Using the shorthand of \eqref{Mal}, we have
so
$$M(\alpha')^{\vee} = (-b_r+ r-1,\dots, -b_1, -c_{s-1}+n-2, \dots, -c_1+r).$$
We consider the $2n(n-1)$-dimensional real Hodge structure
$$M(\alpha,\alpha') = R_{\CC/\RR}M_\CC(\al)\otimes M_\CC(\alpha')^{\vee}.$$
Then
$M(\alpha,\alpha')_\CC$ is the sum of eigenspaces of the form
$$(a_i + n-i - b_k +k-1, \bullet);   (a_i + n-i - c_j + j - 1 + r, \bullet),$$
where in each case the two integers in the ordered pair add up to $2n-3$.  The space $M(\alpha,\alpha')_\CC$ contains an $n(n-1)$-dimensional subspace $F^+M(\alpha,\alpha')$, defined as in \cite{H13}:  it  consists of pairs $(x,y)$ as above with $x > y$.

\subsection{Igusa towers and pairings}

Let $p$ be a prime that splits in $\CK$ as the product $\fp \cdot \fp'$.   Identifying the algebraic closures of $\QQ$ in $\CC$ and in $\Qpb$ places the embeddings of $\CK$ in $\CC$ and in $\Qpb$ in bijection.   We let $\fp$ be the prime above $p$ associated to the fixed embedding $\iota:  \CK \hookrightarrow \CC$ and identify 
\begin{equation}\label{UGL} U(V)(\Qp) \isoarrow GL(n,\CK_\fp) \isoarrow GL(n,\Qp)
\end{equation}  
in such a way that $\fp^-_y \oplus \fk_y$ is identified with the Lie algebra of an upper triangular parabolic subalgebra of $Lie(GL(n))$.  
We also denote by $incl_p:  \CK \hookrightarrow \Qpb$ the embedding corresponding to $\fp$.  
We fix a neat level subgroup $K \subset U(V)(\af)$ with $K = K_p\times K^p$ with $K_p = GL(n,\Zp)$.  The Shimura variety  ${}_KSh(U(V))$ then has a smooth model ${}_KS(V)$ as a moduli space (Shimura variety of abelian type) over $Spec(\CO)$ for some finite $\Zp$-algebra $\CO$.  For each $\kappa$ as above the vector bundle $\CE_\kappa$ extends to a vector bundle over ${}_KS(V)$.  

We choose $K$ so that $K \cap U(V')(\af) = K'$ is neat and admits a factorization $K' = GL(n-1,\Zp)\times K^{\prime,p}$.   We define ${}_{K'}S(V')$ as in the previous paragraph, and assume the embedding
\begin{equation}\label{emb}   {}_{K'}S(V') \hookrightarrow {}_{K}S(V) 
\end{equation}
restricts (see \S \ref{igusa} below) to an embedding of ordinary loci
\begin{equation}\label{embo}   {}_{K'}S(V')^{ord} \hookrightarrow {}_{K}S(V)^{ord} 
\end{equation}
which lifts to a morphism of Igusa towers
\begin{equation}\label{emboIg}   {}_{K'}Ig(V') \hookrightarrow {}_{K}Ig(V). 
\end{equation}

\subsubsection{Embeddings of Igusa towers}\label{igusa} As in \cite{EHLS}, we use the theory of ordinary Hida families developed in Hida's book \cite{Hida} (and completed by Kai-Wen Lan's verification of the necessary conditions:  see the discussion in \cite[\S 2.9.6]{EHLS}).  This theory is based on the study of analytic functions on Igusa towers.  In this paper we use the conventions of \S 2 of \cite{EHLS}.   We choose a $p$-adic embedding $\iota_p:  \CK \hookrightarrow \CC_p$ as
in \cite[\S 1.4.1]{EHLS}, so that $\iota_p$ and the chosen inclusion $\iota: \CK \hookrightarrow \CC$ are associated as in \S 1 of \cite{HLS}.   

In order to define cohomological pairings between $p$-adic modular forms on the Shimura varieties ${}_{K'}S(V')$ and ${}_{K}S(V)$ we need to know that the map \eqref{embo} actually exists.  In the first place, strictly speaking there is a map of Shimura data
\begin{equation}\label{plus1}
(U(V'),Y_{V'})\times (U(1),Y_1^+) \hookrightarrow (U(V),Y_V).
\end{equation}
The second factor on the left is a (pro)-finite set without any additional arithmetic structure -- recall that with our conventions the homomorphism $Y_1^+$ is trivial.  To understand the map \eqref{embo} it is nevertheless better to start with the map of Shimura data of PEL type
\begin{equation}\label{plus2} (G(U(V')\times U(1)),X_{V'}^{''}) \hookrightarrow (GU(V),X_V))
\end{equation}
with $X_{V'}^{''}$ defined as in \cite{H19}, \S 2.2.  As a reminder:  we can also embed $G(U(V')\times U(1))$ in $GU(V') \times GU(1)$, with 
$GU(1) = R_{\CK/\QQ} (\Gm)_\CK$. Then $X_{V'}^{''}$ is a $G(U(V')\times U(1))(\RR)$-conjugacy class of homomorphisms whose image under the embedding
in $GU(V') \times GU(1)$ lies in the product $X_{V'} \times X_{0,1}$, where $X_{0,1}$ is the homomorphism 
$R_{\CC/\RR} (\Gm)_\CC \ra GU(1)(\RR)$ whose value on $\RR$-valued points is given by $z \mapsto \bar{z}$.

 Now the map \eqref{plus2} defines a morphism of PEL Shimura varieties, and thus of smooth models in level $K = K_p \times K^p$ as above:
 \begin{equation}\label{embplus}   {}_{K'}S(G(U(V')\times U(1)), X_{V'}^{''}) \hookrightarrow {}_{K}S(GU(V),X_V),
\end{equation}
with notation (and level subgroup $K'$) defined by analogy with \eqref{emb}.  We define ordinary loci
$${}_{K_{V'}}S(GU(V'),X_{V'})^{ord} \subset {}_{K'}S(GU(V'),X_{V'});$$ $${}_{K_1'}S(GU(1),X_{0,1})^{ord} \subset {}_{K'}S(GU(1),X_{0,1});$$
$${}_{K'}S(G(U(V')\times U(1)), X_{V'}^{''})^{ord}  \subset {}_{K'}S(G(U(V')\times U(1)), X_{V'}^{''})$$
as well as
$${}_{K}S(GU(V),X_{V})^{ord} \subset {}_{K}S(GU(V),X_V).$$
(Level subgroups are assumed compatible with all morphisms.)
    
We recall the discussion of the Igusa varieties in \cite{HLS}.  For any $n \geq 0$ we can define Igusa coverings -- we omit the prime-to-$p$ level structures
from the notation --
$$GIg(V')_n \ra {}_{K_{V'}}S(GU(V'),X_{V'})^{ord}; GIg(0,1)_n \ra {}_{K_1'}S(GU(1),X_{0,1})^{ord}$$
and
$$GIg(V)_n \ra {}_{K}S(GU(V),X_{V})^{ord}.$$
(We reserve the notation $Ig(V)$ for the Igusa towers over the unitary group Shimura varieties, and $Ig(V)_n$ for the Igusa covering in level $p^n$.)
When $n = 0$ this is the identity map.
These correspond to pairs $(\underline{A}_{V'},j^o_{V'}), (\underline{A}_{0,1},j^o_{0,1}), (\underline{A}_{V},j^o_V)$ as in \cite{HLS}, (2.1.6.2).  Here for example,
$\underline{A}_V$ is a quadruple $(A,\lambda,\iota,\alpha^p)$, with $A$ an abelian scheme of dimension $n$, and
$$j^o_V:  M(V)^0\otimes \mu_{p^m} \hookrightarrow A[p^m]$$ 
is an embedding of finite flat group schemes with $\CO_K/p^m\CO_K$-action.   The free $\CO_K$-submodule $M(V)^0 \subset V$ (resp. $M(V')^0 \subset V'$,
resp. $M(0,1)^0 \subset \CK$) has the property that the action of 
$\CO_K$ is a sum of $r$ copies (resp. $r$ copies, resp. $0$ copies) of $\iota$ (or $\iota_p$) and $s$ copies (resp. $s-1$ copies, resp. $1$ copy) of $c\iota$ (or $c\iota_p$).  

We let $GIg(V',(0,1))_m$ denote the fiber product of $GIg(V')_m\times GIg(0,1)_m$ with ${}_{K'}S(G(U(V')\times U(1)), X_{V'}^{''})$ over 
${}_{K_{V'}}S(GU(V'),X_{V'})^{ord}\times {}_{K_1'}S(GU(1),X_{0,1})^{ord}$.  
With these conventions, it follows as in the discussion in \cite[\S 2.1.1]{HLS} that 
\begin{lemma}  The morphism
\eqref{embplus} defines canonical morphisms of Igusa towers
$$GIg(V',(0,1))_n \hookrightarrow GIg(V)_n$$
for $n \geq 0$.  For $n = 0$ this defines a morphism 
$${}_{K'}S(G(U(V')\times U(1)), X_{V'}^{''})^{ord}  \hookrightarrow {}_{K}S(GU(V),X_{V})^{ord}.$$
\end{lemma}
 
 Finally, the maps \eqref{embo} and \eqref{emboIg} are obtained by twisting with the Igusa tower for the Shimura datum $(GU(1),X_{0,1})$ as in \S 2 of \cite{H19}. We omit the details.
 
 \begin{remark}  The local computations in \cite{EHLS} make it clear that the Euler factors at $p$ in the standard $p$-adic $L$-function for ordinary families depends
 strongly on the signatures at primes above $p$, in a way that is broadly consistent with the conjectures of Coates and Perrin-Riou on $p$-adic $L$-functions for motives.  The same dependence on archimedean data is expected for $p$-adic $L$-functions constructed in the setting of the Ichino-Ikeda-N. Harris Conjecture \ref{conjectureII}. The signature enters in \cite{EHLS} through a twist that guarantees the existence of embeddings of Igusa towers; see Remark 3.1.4 of \cite{EHLS}. It is likely that similar twists will be needed in order to extend the constructions of the present paper to the setting of Pilloni's higher Hida theory \cite{Pi}.  
 
 \end{remark}
 
 Let $(H_1,h_1) \subset (GU(V),X(V))$ be a CM pair -- in other words, $H_1$ is a torus.  We say $(H_1,h_1)$ is an {\it ordinary CM pair} if the image of the morphism 
 $_{K(H_1)}S(H_1,h_1) \rightarrow {}_{K}S(GU(V),X_{V})$ consists of PEL abelian varieties with ordinary reduction at $p$, for appropriate level subgroups.  Thus when $K\cap U(V)(\ad) = K^p \times GL(n,\Zp)$, the morphism of $_{K(H_1)}S(H_1,h_1) \rightarrow {}_{K}S(GU(V),X_{V})$ extends to a finite morphism of integral models if $(H_1,h_1)$ is an ordinary CM pair.  We define an ordinary CM pair $(H,h) \subset (U(V),Y_V)$ analogously.

\subsubsection{Pairings}

Fix $\kappa$ and $\delta_{\kappa,\kappa'}:  \CE_\kappa ~|_{Sh(V')} \ra \CE_{\kappa'}$ as in Proposition \ref{diffops}.  Let $d = rs = \dim Sh(V)$, $d' = r(s-1) = \dim Sh(V')$, and define
$$\CE_{\kappa^{\prime,\flat}}  = \Omega^{d'}\otimes \CE_{\kappa'}^{\vee}$$
be the Serre dual of $\CE_{\kappa'}$.  Then there is a canonical Serre duality pairing
$$H^0({}_KS(V),\CE_{\kappa}) \otimes H^{d'}({}_{K'}S(V'),\CE_{\kappa^{\prime,\flat}}) \ra \CO.$$
More generally, if $K'_{p,r} \subset K'_p$ is the congruence subgroup defined in \cite{EHLS}, $K'_r = K'_{p,r}\times K^{\prime,p}$, we can define a finite flat $\CO$-module
$$H^{0}({}_{K'_r}S(V'),\CE_{\kappa^{\prime}}) \subset H^{0}({}_{K'_r}Sh(V'),\CE_{\kappa^{\prime}}) := H^{0}({}_{K'_r}S(V'),\CE_{\kappa^{\prime}})\otimes_{\CO}\CO[\frac{1}{p}]$$
to be 
\begin{equation}\label{padicclass}
H^{0}({}_{K'_r}S(V'),\CE_{\kappa^{\prime}}) = H^{0}({}_{K'_r}S(V'),\CE_{\kappa^{\prime}})\otimes_{\CO}\CO[\frac{1}{p}]\cap \CV_{V'}
\end{equation}
where $\CV_{V'}$ is the algebra of $p$-adic modular forms on $Sh_{V'}$ (see below).  Then we let 
\begin{equation}\label{intHd} H^{d'}({}_{K'}S(V'),\CE_{\kappa^{\prime,\flat}}) = Hom(H^{0}({}_{K'_r}S(V'),\CE_{\kappa^{\prime}}),\CO).
\end{equation}
and we obtain a Serre duality pairing 
\begin{equation}\label{Serre}
H^0({}_{K_r}S(V),\CE_{\kappa}) \otimes H^{d'}({}_{K'_r}S(V'),\CE_{\kappa^{\prime,\flat}}) \ra \CO.
\end{equation}
where $K_r = K_{p,r}\times K^p$ is defined as before.

\section{$p$-adic modular forms and differential operators}
\subsection{Basic definitions}\label{bdef}
The algebra $\CV_V$ of $p$-adic modular forms on $Sh_V$ is defined as in \S 2.6 of \cite{EFMV}, following \cite{Hida}.  Specifically, we let $B, N, T \subset GL(n)$ denote respectively the upper triangular Borel subgroup, its unipotent radical, and its diagonal torus.
For any pair of non-negative
integers $(n,m)$ we let 
$$Ig_{n,m,V} = Ig(V)_n \times_{Spec(\CO)} Spec(\CO/p^m)$$
where $Ig(V)_n$ is the Igusa covering in level $p^n$, as above.  In the notation of \cite{EFMV} we let
$$V_{n,m,V} = H^0(Ig_{n,m,V}, \CO_{Ig_{n,m,V}});  V_{\infty,m,V} = \varinjlim_n V_{n,m,V};  V_{\infty,\infty,V} = \varprojlim_m V_{\infty,m,V}.$$
and set
$$\CV_V = V_{\infty,\infty,V}^{N(\Zp)},$$
where $N$ is the maximal unipotent subgroup of $U(V)$ defined in \cite[2.1]{EFMV}.

The group $T(\Zp)$ acts on $\CV_V$ and for any algebraic character $\alpha$ of $T$ we let $\CV_V[\alpha] \subset \CV_V$ denote the corresponding
eigenspace; the elements of $\CV_V[\alpha]$ are called $p$-adic modular forms of weight $\alpha$.   There are canonical embeddings
\begin{equation}\label{class}   \Psi = \Psi_\alpha:  H^0(S_V,\CE_\alpha) \hookrightarrow \CV_V[\alpha];
\end{equation}
compatible with multiplication in the sense that
$$\Psi_\alpha \otimes \Psi_\beta = \Psi_{\alpha + \beta}:  H^0(S_V,\CE_\alpha)\otimes H^0(S_V,\CE_\beta) \overset{\times}\to H^0(S_V,\CE_{\alpha+\beta}) 
\hookrightarrow \CV_V[\alpha+\beta];$$
the forms in the image of \eqref{class} are called {\it classical}.  
More generally, if $\alpha:  T(\Zp) \ra \CO_{\CC_p}^\times$ is any continuous character, we may define the space $\CV_V[\alpha] \subset \CV_V\otimes \CO_{\CC_p}$ of $p$-adic modular forms of weight $\alpha$.  In what follows, we use the embeddings $incl_p$ and $\iota$ to identify the maximal tori $T_y$ and $T$, so that $B$ is contained in the maximal parabolic subgroup with Lie algebra $\fp^-_y \oplus \fk_y$.   If $\alpha$ is a classical weight, we write 
\begin{equation}\label{inclusion}
\Psi_\alpha:  \CV_V[\alpha] \subset \CV_V
\end{equation}
for the tautological inclusion, extending the inclusion of \eqref{class}; the notation is consistent.  

The embedding \eqref{emboIg} determines a map
\begin{equation}\label{resV}
res_{V'}:  \CV_V \ra \CV_{V'}.
\end{equation}
This embedding is compatible with action of the torus $T'$ on the two sides through its inclusion in $T$.

 Let $A$ be an algebraic torus over $Spec(\Zp)$.  For any complete $p$-adic algebra $\CO$, define the Iwasawa algebra
 $$\Lambda_\CO(A) = \CO[[A(\Zp)]] = \varprojlim_{U \subset A(\Zp)} \CO[A/U]$$
 where $U$ runs over open compact subgroups of $A(\Zp)$.  Let $\C(A(\Zp),\CO)$ denote the $\CO$-algebra of continuous $\CO$-valued functions on $A(\Zp)$, endowed with the topology defined by the sup norm.
 
 \begin{defn}  A $\CO$-valued $p$-adic measure -- more simply, an $\CO$-valued measure -- on $A$ is a continuous $\CO$-homomorphism from $\C(A(\Zp),\CO)$ to $\CO$.
\end{defn}

It is well known that the set of $\CO$-valued measures on $A$ forms an $\CO$-Banach module that is naturally identified with $\Lambda_\CO(A)$.  Multiplication in the $\CO$-algebra $\Lambda_\CO(A)$ corresponds to {\it convolution} of measures.
If $\varphi \in \C(A(\Zp),\CO)$ and $\mu \in \Lambda_\CO(A)$, we write
$$\int_{A(\Zp)} \varphi d\mu := \mu(\varphi).$$

For any torus $A$ over $Spec(\Zp)$, and any $\Zp$-algebra $\CO$, let 
$$\CW_\CO(A) = Hom_{cont}(A(\Zp),\CO^{\times}) = Hom_{cont}(\Lambda_\CO(A),\CO^{\times}).$$
The {\it weight space} for $A$ is the rigid analytic space over $\Qp$ attached to $\Lambda_\CO(A)$.  A weight for $A$ is then an element of $\CW_{\CO}(A)$.

When $\CO = \CV_V$ we write $Meas(A,\CV_V)$ instead of $\Lambda_{\CV_V}(A)$.   


\subsection{$p$-adic differential operators}\label{pdiffops}  

There is a quotient $\Ts$ of the torus $T_y$, of rank $min(r,s)$, defined by a sublattice of the lattice of characters of $T_y$:  the characters of $\Ts$ are spanned by the ones called {\it symmetric} in Definition 2.4.4 of \cite{EFMV}.  Symmetric characters are also assumed to be {\it dominant}; the precise condition is recalled below.  

We recall the normalization of $C^\infty$ differential operators (Maass operators) from \cite{EFMV}, \S 3.3.1.  For a weight $\kappa$ of holomorphic type we let
$\CE_\kappa(C^\infty)$ denote the space of $C^\infty$ global sections of $\CE_\kappa$.  Let $\lambda$ be a symmetric character of $T_y$ and let 
\begin{equation}\label{maass}
D_\kappa^{\lambda}:  \CE_\kappa(C^\infty) \ra \CE_{\kappa+\lambda}(C^\infty)
\end{equation}
be the differential operator introduced on pp. 467-468 of \cite{EFMV} (we are writing weights additively rather than multiplicatively).   For any weight $\alpha$ of $T$ let $[\alpha]'$ denote its restriction to the subtorus $T' \subset T$.  Let
$$R^{\infty}_{V,V'}:  \CE_{\alpha}(C^\infty) \ra \CE_{[\alpha]'}(C^\infty)$$ denote 
the restriction of $C^\infty$ sections (any $\alpha$).  
We denote 
$$pr^{hol}_{[\alpha]'}:  \CE_{[\alpha]'}(C^\infty) \ra H^0(Sh(V'),\CE_{[\alpha]'})$$
denote the orthogonal projection on holomorphic sections (any $\alpha$).  

 Let $\kappa' = [\kappa + \lambda]'$.    
The relation between the $D_\kappa^\lambda$ and the holomorphic operator $\delta_{\kappa,\kappa'}$ is given by the following

\begin{lemma}\label{pluridecomp}  

We write 
$$D^{hol}(\kappa,\kappa^\dag) = pr^{hol}_{[\kappa^\dag]'}\circ R^{\infty}_{V,V'}\circ D_\kappa^{\kappa^\dag - \kappa}$$
Then for all $\kappa^{\dag} \leq \kappa'$ there exist unique elements $\delta(\kappa',\kappa^{\dag}) \in U(\fp^{+,\prime})$, 
defined over $\CK$, such that 
$$D_\kappa^\lambda = \sum_{\kappa^{\dag} \leq \kappa'} \delta(\kappa',\kappa^{\dag})\circ D^{hol}(\kappa,\kappa^\dag).$$
The term
$\delta(\kappa',\kappa')$ is a non-zero scalar in $\CK$.  
\end{lemma}
\begin{proof}  This is the analogue of Corollary 4.4.9 of \cite{EHLS} and is proved in the same way.   
\end{proof}

The idea of the proof is roughly the following.   Write $E_{\kappa,y}$ for the fiber at $y$ of the pullback of $\CE_\kappa$ to the symmetric space $Y_V$; this
is an irreducible representation of $K_y$.  Then
$D_\kappa^{\kappa^\dag - \kappa}$ lifts, on automorphic forms, to a differential operator given in the enveloping algebra 
of $\fp^+_y$ by an explicitly normalized projection onto the $\kappa^\dag$-isotypic subspace of 
$$E_{\kappa,y}\otimes Sym^{|\kappa^\dag - \kappa|}(\fp^+_y),$$
where $|\kappa^\dag - \kappa|$ is the degree of the weight $\kappa^\dag - \kappa$.  This isotypic subspace is the sum of its intersections with 
the irreducible constituents of the restriction to $U(\fg')$ of the discrete series $\DD_\kappa$, as in \eqref{reskappa}.  Only one of these
intersections is the highest $K'_y$-type subspace of its corresponding constituent; this is the image of $pr^{hol}_{\kappa'}$.  Each of the others
is obtained from the highest $K'_y$-type of its irreducible $U(\fg')$-constituent $\DD_{\kappa^{\dag}}$.  The existence of $\delta(\kappa',\kappa^{\dag})$
as in the lemma then follows from the obvious fact that $\DD_{\kappa^{\dag}}$ is generated over $U(\fp^{+,\prime})$ by its highest $K'_y$-type subspace.

The analogous $p$-adic differential operators are constructed in \cite[\S 3.3.2]{EFMV}.  To preserve some of their notation while avoiding ambiguity we write
$$\CE_\kappa(ord) = H^0(Ig_V,\CE_\kappa).$$
Then the operators are denoted
\begin{equation}\label{pmaass}
D_\kappa^{\lambda,ord}:  \CE_\kappa(ord) \ra \CE_{\kappa+\lambda}(ord).
\end{equation}

We define a $p$-adic character $\chi$ of $\Ts$ to be a continuous group homomorphism $\Ts(\Zp) \ra \ZZ_p^\times$ that arises as the $p$-adic limit of dominant characters $\lambda$.
The main results of \cite{EFMV} are summarized in the following theorem:
\begin{thm}\label{pdiff}   (a) For any dominant character $\lambda$ of $\Ts$ (or any symmetric character $\lambda$ of $T_y$) there is a $p$-adic differential operator
\begin{equation}\label{adiff} \Theta^\lambda:  \CV_V \ra \CV_V \end{equation}
characterized uniquely by either of the following properties:
\begin{itemize}
\item[(i)] For all classical weights $\alpha$,
\begin{equation}\label{diffclass}
\Theta^\lambda\circ \Psi_\alpha = \Psi_{\alpha+\lambda}\circ D_\alpha^{\lambda,ord}.
\end{equation}
Here $\Psi_{\alpha+\lambda}$ is understood in the sense of \eqref{inclusion}.
\item[(ii)]  Let $\alpha$ be algebraic.  Let $j:  (H,h) \ra (U(V),Y_V)$ be an ordinary CM pair, and for any $\kappa$ let 
$$R_{H,h,j,\kappa}:  H^0(S(V),\CE_\kappa) \ra H^0(S(H,h),j^*\CE_\kappa)$$
 denote the restriction map.  Let 
 $$R_{H,h,j,\kappa}^p:  \CE_\kappa^{ord} \ra H^0(S(H,h),j^*\CE_\kappa);$$
 resp.
  $$R_{H,h,j,\kappa}^\infty:  \CE_\kappa(C^\infty) \ra H^0(S(H,h),j^*\CE_\kappa);$$
 denote the analogous restrictions on $p$-adic, resp. $C^\infty$, modular forms.  Then for any $F \in H^0(S(V),\CE_\alpha)$,
 \begin{equation}\label{diffclassp}
 R_{H,h,j,\alpha+\lambda}^p\circ \Theta^\lambda\circ \Psi_\alpha(F) = R_{H,h,j,\kappa}^\infty\circ D_\alpha^\lambda(F).
  \end{equation}
\end{itemize}



(b) For any $p$-adic character $\chi$ of $\Ts$ there exists a $p$-adic differential operator
$$\Theta^\chi:  \CV_V \ra \CV_V$$
characterized by the property:  whenever $\chi$ can be written as $\lim_i \lambda_i$, where $\lambda_i$ are dominant algebraic characters, satisfying the inequalities of Theorems 5.2.4 and 5.2.6 of \cite{EFMV}, then 
$$\Theta^\chi = \lim_i \Theta^{\lambda_i}$$
(limit in the operator norm).  

(c) If $F \in \CV_V$ is a $p$-adic modular form of weight $\alpha \in X^{an}(T_y)$, then $\Theta^\lambda(F)$ is a $p$-adic modular form of weight $\alpha + \lambda$.

\end{thm}

\begin{proof}    Parts (a)(i), (b), and (c) are contained in Corollary 5.2.8 of \cite{EFMV}.   Part (a)(ii) can be proved by the arguments quoted in the proof of \cite[Proposition 7.2.3]{EFMV}.  A complete proof will be supplied in forthcoming work.
\end{proof}


\begin{remark}\label{idempotent}  The inequalities cited in the statement of Theorem \ref{pdiff} (b) guarantee that the characters $\lambda_i$ tend to infinity in the positive chamber; indeed, that for every positive root $\alpha$, $\lim_i <\alpha,\lambda_i> = \infty$.  In particular, when $\chi = 1$ is the trivial character, $\Theta^1 := \Theta^\chi$ is not the identity operator on $\CV_V$, though it is an idempotent.    This is familiar from Hida's theory in the case of elliptic modular forms:  the p-adic differential operator of non-integral weight $\chi$  multiplies the $n$-th Fourier coefficient of a classical modular form by the power $n^\chi$, which is only defined if $(p,n) = 1$.  A classical modular form whose $n$th Fourier coefficient vanishes for every $n$ divisible by $p$ is called {\it $p$-depleted}.  In our situation, the operation $F \mapsto \Theta^1(F)$ can be understood as $p$-depletion, even when the unitary group (over a general totally real field) is anisotropic.  
\end{remark}



\section{One-dimensional $p$-adic measures defined by a holomorphic automorphic form}
The differential operators defined in \S \ref{pdiffops} give rise to a $p$-adic measure.  We believe that they can be used to define a measure on the full space $\Ts(\Zp)$, but for the purposes of this paper we restrict our attention to a $1$-dimensional quotient torus, since the necessary definitions are already in \cite{EFMV} in the form we need.   First, we state a corollary to Theorem \ref{pdiff}:

\begin{cor}\label{measure} Let $F \in \CV_V$ be a $p$-adic modular form of weight $\alpha$. Then there exists a $\CV_V$-valued measure $\mu^*_F$ on $\Ts(\Zp)$ characterized by the property that, for any $p$-adic character $\chi$ of $\Ts$, viewed as a symmetric character of $T_y$, we have
$$\int_{\Ts(\Zp)}  \chi d\mu^*_F = \Theta^{\chi-\alpha}(F).$$
\end{cor}

We recall that $T_y$ is a maximal torus of the group $GL(r)\times GL(s) \isoarrow K_y$, and that the adjoint action of $K_y$ on $\fp^+_y$ is equivalent to the natural conjugation action on the space of $r \times s$ matrices.  This action is identified in \cite{EFMV} with the representation $St_r \otimes St_s$, where $St_a$ is the standard representation of $GL(a)$ on $a$-dimensional space.  Then the symmetric algebra 
\begin{equation}\label{symm} Sym^*(\fp^+_y/\fp^{+,\prime}_y) \isoarrow \oplus_{i \geq 0}Sym^*((St_r\otimes St_s)/(St_r\otimes St_{s-1}) \isoarrow \oplus_{i \geq 0}Sym^*((St_r\otimes St_1) 
\end{equation}
where the last isomorphism is given by the isotypic decomposition $St_s \isoarrow St_1 \oplus St_{s-1}$ as representation of the standard Levi subgroup $GL(1)\times GL(s-1) \subset GL(s)$.   The dominant characters $\lambda$ of $\Ts$ can be written as parameters \eqref{rstuple}
$$  (b_1 \geq b_2 \geq \dots \geq b_s \geq 0 \dots 0; b_1 \geq b_2 \geq \dots \geq b_s)$$
if $r \geq s$, and with the $0$'s in the second half of the parameter if $s > r$.   Then the representations occurring in \eqref{symm} have parameters
\begin{equation}\label{lambdab}  \lambda_b = (b \geq 0 \geq \dots \geq 0; b; 0 \geq \dots \geq 0),
\end{equation}
where the two colons separate parameters for $GL(r)\times GL(1)\times GL(s-1)$.   

If $b \in \Zp$, we write $\lambda_b = \lim_i b_i$ where $b_i = (b_{1,i},\dots, b_{min(r,s)_i})$ where all the $b_{j,i}$ are non-negative integers, 
$b = \lim_i b_{1,i}$ in the $p$-adic topology, $\lim_i b_{j,i} = 0$ in the $p$-adic topology for $j > 1$, and for all $1 \leq j \leq min(r,s)$, $\lim_i b_{j,i} = \infty$ in the real topology.

Let $X(\Ts)$ denote the character lattice of $\Ts$.  Let $X_1 \subset X(\Ts)$ be the characters of the form $\lambda_b$ as in \eqref{lambdab}.  Then $X_1$ is the character group of a  $1$-dimensional quotient of $\Ts$, which we identify with $GL(1)$.   Restricting the measure $\mu^*_F$ to characters of $GL(1)$, we obtain the corollary:

\begin{cor}\label{measure1}  Let $F \in \CV_V$ be a $p$-adic modular form of weight $\alpha$. Then there exists a $\CV_V$-valued measure $\mu_F$ on $GL(1,\Zp)$ characterized by the property that, for any $p$-adic integer $b$, we have
$$\int_{GL(1,\Zp)}  x^b d\mu_F = \Theta^{\lambda_b}(F).$$
\end{cor}

\begin{definition}\label{equivmeas} Let $A$ be a torus over $Spec(\Zp)$.  We say the $\CV_V$-valued measure $\mu$ on $A(\Zp)$ is equivariant of weight $\alpha$ if for any character $\chi$ of $A$, the integral 
$\int_{A(\Zp)}  \chi d\mu$ is a $p$-adic modular form of weight $\chi + \alpha$ for some fixed weight $\alpha$.
\end{definition}

The following corollary is then a consequence of Theorem \ref{pdiff} (c).
\begin{cor}\label{equivF}  Let $F\in \CV_V$ be a $p$-adic modular form of weight $\alpha$.  Then the measures $\mu^*_F$ (resp. $\mu_F$) on $\Ts(\Zp)$ (resp. $GL(1,\Zp)$) are equivariant of weight $\alpha$.
\end{cor}

We will be pairing the measure $\mu_F$ -- or rather its restriction  to $Sh_{V'}$ -- with Hida families of ordinary $p$-adic modular forms on $U(V')$.  We could also pair the $\dim \Ts$-parameter measure with Hida families, but they will not give rise to more general special values, because the differential operators on $Sh_V$ in directions parallel to $Sh_{V'}$ do not change the automorphic representation of $U(V')$.  

Suppose now that $F \in \CV_V$ is a classical form of weight $\kappa$.  Let $\kappa'$ satisfy the inequalities of Lemma \ref{param}, so there is a holomorphic differential operator $\delta^{\kappa,\kappa'}$ as in Proposition \ref{diffops}.   

\begin{lemma}\label{decomp}  (a)  For all $\kappa^\dag$ that satisfy the inequalities of Lemma \ref{param}, there is a differential operator
$$\theta^{hol}(\kappa,\kappa^\dag):  \CV_V \ra \CV_V$$ 
such that 
$$res_{V'}\circ\theta^{hol}(\kappa,\kappa^\dag)(F) = \delta^{\kappa,\kappa'}(F).$$
(b)  For all $\kappa^\dag \leq \kappa'$,  there are differential operators
$\theta(\kappa,\kappa^\dag):  \CV_V \ra \CV_V$ such that
$$\theta(\kappa,\kappa') = \sum_{\kappa^{\dag} \leq \kappa'} res_{V'} \circ \theta(\kappa,\kappa^{\dag})\circ \theta^{hol}(\kappa,\kappa^{\dag}),$$
with $\theta(\kappa,\kappa)$ a non-zero scalar.  Here $res_{V'}$ is as in \eqref{resV}.  
\end{lemma}

\begin{proof} Part (a) is the analogue of Proposition 8.1.1 (d) of \cite{EHLS}; it is derived in the same way from properties of restriction to CM points -- in this case 
from Theorem \ref{pdiff} (a)(ii).   Part (b) is then the analogue of \cite[Corollary 8.1.2]{EHLS}.
\end{proof}

In what follows, the terms {\it anti-ordinary} and {\it anti-holomorphic} are used as in \cite{EHLS}; these are reviewed in Appendix \ref{antiord}.  

\begin{proposition}\label{tkk}  
There is a $p$-adic differential operator $\theta^{\kappa,\kappa'}:  \CV_V \ra \CV_{V'}$ with the property that, for any anti-holomorphic anti-ordinary automorphic form $g$ of weight $\kappa'$ on $U(V')$ and any holomorphic automorphic form $F$ of weight $\kappa$ on $U(V)$, we have
$$[\theta^{\kappa,\kappa'}(F),g] = [\delta^{\kappa,\kappa'}(F),g].$$
\end{proposition}
\begin{proof}   This follows from Lemma \ref{decomp} and from estimates on the denominators used to define the ordinary projector, as in the proof of 
\cite[Proposition 8.1.3]{EHLS}.
\end{proof}




In what follows we assume $F$ to belong to a fixed holomorphic automorphic representation $\pi$.

\section{Hida families}\label{Hidafamm}

Recall the maximal torus $T' \subset U(V')$.  Let $\Lambda' = \Lambda_{\CO}(T')$ be the Iwasawa algebra of $T'(\Zp)$; it is a noetherian local ring that is non canonically isomorphic to the tensor product over $\CO$ of $n-1$ copies of $\CO[[1+p\ZZ_p]]$.  

\subsection{Ordinary parameters}\label{ordparam}  

We let $\TT$ denote the ordinary $p$-adic Hecke algebra for cusp forms on the group $U(V')$.  
For the purposes of this paper, a Hida family is determined by a single anti-ordinary anti-holomorphic automorphic representation $\tau$ of $U(V')$, and the completion $\TT_\tau$ of $\TT$ at the maximal ideal $\fm_\tau$ corresponding to $\tau$.  
We write $\CO$ for the coefficient ring denoted $\CO_\tau$ as in \S 7.3 of \cite{EHLS}, so that $\Lambda' = \Lambda_{\Zp}(T')\otimes \CO$, where $\Lambda_{\Zp}(T')$ is the  Iwasawa algebra of weights for $T'_y$.  We let $\mathbf{T}^{\ord}_{K'_r,\kp,\CO}$ denote the ordinary Hecke algebra of weight $\kp$ and level $K'_r$ -- notation as  in \cite[\S 6.6.6]{EHLS} --
and let $\TT_{r,\kp}$ denote the completion of $\mathbf{T}^{\ord}_{K'_r,\kp,\CO}$ at $\fm_\tau$; the representation $\tau$ will be 
fixed for the remainder of the paper.   We can write
$$\TT_\tau = \varprojlim_r \TT_{r,\kp,\tau}$$ 
for any sufficiently regular $\kappa'$, as in \cite[\S 7.1, Theorem 7.1]{EHLS}.   For all claims regarding Hida families we refer to Hida's book \cite{Hida}.  In particular,  $\TT_\tau$ is a finite flat $\Lambda'$-algebra \cite{Hida}.

\subsection{The Gorenstein condition}

We introduce two versions of the Gorenstein hypothesis that was used in \cite{EHLS} to define $p$-adic $L$-functions with values in Hida's ordinary Hecke algebra.  The first is adapted to (holomorphic) automorphic forms of fixed weight $\kappa'$; compare \cite[Definition 6.7.9]{EHLS}:

\begin{definition}\label{Gorfin}  
The $\TT_{r,\kp}$-module $S^{\ord}_{\kp}(K'_r,\CO)_{\tau}$ is said to satisfy
the {\it Gorenstein hypothesis} if the following conditions hold.
\begin{itemize}
\item  $\TT_{r,\kp} \isoarrow \hat{\TT}_{r,\kp} := Hom_\CO(\TT_{r,\kp},\CO)$ as $\CO$-algebras.
\item  $S^{\ord}_{\kappa'}(K'_r,\CO)_{\tau}$ is free over $\TT_{\kp}$.
\end{itemize}
The $\bT^{\ord}_{K'_r,\kp,\CO}$-module $S^{\ord}_{\kp}(K'_r,\CO)$ is said to satisfy
the Gorenstein hypothesis if all its localizations at maximal ideals
of $\bT_{K'_r,\kp,\CO}$ satisfy the two conditions above.
\end{definition}

The second version is a hypothesis on the big ordinary Hecke algebra, which is a finite $\Lambda'$-algebra; compare \cite[Hypothesis 7.3.2]{EHLS}.

\begin{hyp}\label{gor}{(Gorenstein Hypothesis)}  Let $\hat{\TT}_{\tau} = Hom_{\Lambda'}(\TT_{\tau^\flat},\Lambda')$.
Then 
\begin{itemize}
\item $\hat{\TT}_{\tau}$ is a free rank-one $\TT_{\tau}$-module via the isomorphism $\flat:  \TT_{\tau} \isoarrow \TT_{{\tau^\flat}}$.
\item For each $r$, let $\TT_{\tau}$ act on $Hom_{\CO}(S^{\ord}_{\kappa'}(K^{\prime,p}_r,\CO),\CO)_{\fm_{\tau}}$ by the natural
action twisted by $\flat$.  Then 
$$\hat{S}^{\ord}_\tau = Hom_{\CO}(\varinjlim_r S^{\ord}_{\kappa'}(K^{\prime,p}_r,\CO),\CO)_{\fm_{\tau}}$$ 
(for any sufficiently regular $\kappa'$) is a free
$\TT_{\tau}$-module.
\end{itemize}
\end{hyp}

The hypothesis in Definition \ref{Gorfin} follows for sufficiently regular $\kappa'$ from Hypothesis \ref{gor} and Hida's control theorem.
More precisely, let
\begin{equation}\label{hatS}
\hat S_{\kappa'}^\ord(K;\CO) = Hom_\CO(S_{\kappa'}^\ord(K;\CO),\CO),
\end{equation}
and define $\hat S_{\kappa'}^\ord(K;\CO)_\tau$ analogously.
Then Hida's control theorem (see \cite[Theorem 7.3.1]{EHLS}) asserts, in the present notation, that
\begin{equation}\label{control}
\TT_{\tau}\otimes_{\Lambda'}\Lambda'/I_{\kappa'} \isoarrow (\bT^{\ord}_{K,\kappa,\CO})_{\tau}
\end{equation}
for sufficiently regular $\kappa'$.  Under Hypothesis \ref{gor}, there is then an isomorphism
\begin{equation}\label{duality}
\hat S_{\kappa'}^\ord(K;\CO)_{\tau} \isoarrow S_{\kappa'}^\ord(K;\CO)_\tau
\end{equation}
of dual free $\TT_{\kappa'}$-modules.  We introduce compatible bases of these modules in the next section.

\subsubsection{Bases}\label{bases}  We let
\begin{equation}\label{Om}  
\Omega_\tau = Hom_{\TT_\tau}(\hat{\TT}_\tau,\TT_\tau)
\end{equation}
\begin{equation}\label{Omk}  
\Omega_{\kappa',\tau} = Hom_{\TT_\kp}(\hat{\TT}_\kp,\TT_\kp)
\end{equation}

Under the Gorenstein hypotheses, $\Omega_\tau$ (resp $\Omega_{\kappa',\tau}$)
is a free rank $1$ $\TT_\tau$- (resp. $\TT_\kp$-) module.  In particular, the set of $\TT_\tau$-isomorphisms between
$\TT_\tau$ and $\hat{\TT}_\tau$ (resp. between $\TT_\kp$ and $\hat{\TT}_\kp$) is a torsor under $\TT_\tau^\times$
(resp. $\TT_\kp^\times$).  

\subsection{Ramified local components}\label{local}

We let $\pi$ be the automorphic representation of $G_V$ corresponding to the holomorphic modular form $F$.  Let $\tau$ be as in the previous section, and let $S$ denote the finite set of finite primes  $v$, not including $p$, such that 
either $\pi_v$,  $\tau_v$, or $\CK/\QQ$ is ramified.  In \cite[\S 7.3.4]{EHLS} we introduce a free $\CO$-lattice
$$\hat{I} = \hat{I}_{\tau} \subset (\bigotimes_{v \in S} ~\tau_v)^{K^p}$$ 
with the property that
\begin{itemize}
\item[(a)]  For all sufficiently regular $\kappa'$, there is an isomorphism of $\TT_{\kappa',\tau}$-modules $\hat S_{\kappa'}^\ord(K;\CO)_{\tau} \isoarrow \TT_{\kappa',\tau}\otimes_\CO \hat{I}$;
\item[(b)]  And there is an isomorphism of $\TT_\tau$ modules $\hat{S}^{\ord}_\tau \isoarrow \TT_{\tau}\otimes_\CO \hat{I}$, compatibly with (a) and with the control isomorphisms \eqref{control}.
\end{itemize}

(In \cite{EHLS} the group is $GU(V)$ rather than $U(V')$ and the notation is $\pi^\flat$ rather than $\tau$, but the lattice is defined in the same way.)   

Let $\LL = Frac(\CO)$, $\bar{\LL}$ its algebraic closure.  For any sufficiently regular $\kp$, there is a decomposition 
\begin{equation}\label{congtau}
\TT_{\kappa',\tau}\otimes_\CO \bar{\LL} \isoarrow \bigoplus_{\pi' \in [\tau]_\kp} \TT_{\pi'}
\end{equation}
The elements of $[\tau]_\kp$ are the (anti-ordinary) antiholomorphic automorphic representations $\pi'$ of weight $\kp$ whose Hecke eigenvalues at places outside $S$ are congruent to those of $\tau$; the action of $\TT_{\kappa',\tau}$ on the vectors in each $\pi'$   
factors through the corresponding component $\TT_{\pi'}$.  

We let $K'_S =  K^p \cap \prod_{v \in S} G_{V'}(\QQ_v)$ and assume $K'_S$ admits a factorization $\prod_{v \in S} K'_v$.  We have implicitly been assuming that
the finite parts of our automorphic representations are defined over $\LL$ (see Remark \ref{rationality}).  Let
$$\hat{I}_\LL = \hat{I}\otimes_\CO \LL = \bigotimes_{v \in S} ~(\tau_v)^{K'_v}(\LL).$$
This is naturally the tensor product over $v \in S$ of {\it irreducible} representations of the local (ramified) Hecke algebra $\CH_\LL(G_{V'}(\QQ_v),K'_v)$ of compactly supported $\LL$-valued $K'_v$-biinvariant functions on $G_{V'}(\QQ_v)$.  Let
$\CH_\CO(G_{V'}(\QQ_v),K'_v) \subset \CH_\LL(G_{V'}(\QQ_v),K'_v)$ denote the subalgebra of $\CO$-valued functions and let
$\CH_S = \otimes_v \CH_\CO(G_{V'}(\QQ_v),K'_v)$.
We  make the following simplifying hypothesis:

\begin{hyp}[Local minimality]\label{mini}  For each $v \in S$, there is an $\CO$-lattice $\hat{I}_v \subset (\tau_v)^{K'_v}(\LL)$, invariant under 
$\CH_\CO(G_{V'}(\QQ_v),K'_v)$, with the property that $\hat{I} \isoarrow \otimes_{v \in S} \hat{I}_v$ and
$$\hat{M}^0_{\tau} := \hat{M}^0_{\tau,\hat{I}} := Hom_{\CH_S}(\hat{I},\hat{S}^{\ord}_\tau)$$
is a free rank $1$ $\TT_{\tau}$-module.
\end{hyp}

We let
\begin{equation}\label{minix} \hat{M}_{\tau} := Hom_{\TT_{\tau}}(\hat{M}^0_\tau, \TT_\tau).
\end{equation}
Under Hypothesis \ref{mini}, $\hat{M}_{\tau}$ is a free rank $1$ $\TT_\tau$-module.

Recall that each $v \in S$ is of characteristic prime to $p$, so we can apply the methods of the mod $p$ representation theory of $G_{V'}(\QQ_v)$ \cite{V}.  In particular, one can define the reduction $\bar{\tau}_v$ of each $\tau_v$ for $v \in S$ modulo the maximal ideal of $\CO$, as a semisimple representation of $G_{V'}(\QQ_v)$ of finite length. It is easy to see that Hypothesis \ref{mini} is automatic if each $\bar{\tau}_v$ is irreducible.  In particular, by the theory of \cite{V}, this holds if $p$ is banal for  all $G_{V'}(\QQ_v)$ with $v \in S$, and in particular if $p$ is sufficiently large.  For $v$ split in $\CK$, the condition can be read off the $p$-adic Galois representation attached to $\tau_v$ by the local Langlands correspondence, and corresponds to the usual hypothesis in Galois deformation theory that the Galois representations attached to $\tau'$ congruent to $\tau$ modulo $p$ are  minimally ramified at such $v$.  This is probably also the case for $v$ inert or ramified in $\CK$, but as far as I know this has not been verified.   

The notation $\hat{I}$ is deleted for the sake of legibility.  
By \eqref{control}, Hypothesis \ref{mini} implies that, for all sufficiently regular $\kappa'$,
\begin{equation}\label{minik}
\hat{M}^0_{\kappa',\tau} :=Hom_{\CH_S}(\hat{I},\hat{S}_{\kappa'}^\ord(K;\CO)_{\tau})
\text{ is a free rank $1$ $\TT_{\kappa',\tau}$-module. }
\end{equation}


 An element $f'_S = \sum_j \otimes_{v\in S} f'_{v,j} \in \hat{I}$, with $f'_{v,j} \in \tau_v$, is called {\it primitive} if it generates a $\CO$-direct summand of $\hat{I}$.    Choose a generator $\hat{m}$ of $\hat{M}_{\tau}$ as $\TT_\tau$-module and a primitive $f'_S \in \hat{I}$.  Then for every sufficiently regular $\kp$, $\hat{m}^{-1}(f'_S)$ generates $S_\kp^\ord(K;\CO)_\tau$ over $\TT_{\kp,\tau}$; it defines a linear combination
 \begin{equation}\label{fpip}
\hat{m}^{-1}(f'_S) = \sum_{\pi' \in [\tau]_\kp} (f_{\pi'})  = f'_S\otimes [\sum_{\pi' \in [\tau]_\kp} f^{\prime,S}_{\pi'}] 
:= f'_S\otimes f^{\prime,S}_{\hat{m}}
 \end{equation}
 with notation as in \eqref{congtau}, where $f_{\pi'} = f'_S\otimes f^{\prime,S}_{\pi'} \in \pi'$ with  $f^{\prime,S}_{\pi'}$ a Hecke eigenvector.  
 
 Note that each individual $f_{\pi'}$ is not necessarily integral, but the sum $f^{\prime,S}_{\hat{m}}$ is a divided congruence and is defined over $\CO$.
As we have noted, under the Gorenstein hypothesis the set of generators $\hat{m}$ of $\hat{M}_{\tau}$ is a torsor under 
$\TT_\tau^\times$.   If $\mathbf {t} \in \TT_\tau^\times$, then 
$$f^{\prime,S}_{\mathbf {t}\cdot\hat{m}} = \mathbf {t}\cdot f^{\prime,S}_{\hat{m}}.$$


\section{Contraction of $p$-adic measures with Hida families}\label{contraction}

We fix a ring $\CO$ of integers in a finite extension of $\Qp$, and we assume $\CO$ is a subalgebra of the algebra $\CV_V$ of $p$-adic modular forms.  For a $p$-adic torus $A$ we define $Meas(A,\CV_V)$ as in \S \ref{bdef}.   We choose a collection of congruence subgroups
$$A(\Zp) \supset A_1 \supset A_2 \dots \supset A_r \supset \dots$$
with $\cap_i A_i = \{1\}$, and we let 
$$C_r(A,\CO)  \subset C(A(\Zp),\CO)$$ 
be the $\CO$-submodule of $\CO$-valued functions on $A(\Zp)/A_r$.  
For any algebraic character $\chi$ of $A$ we consider $C_r(A,\CO)\chi$ as a finite rank $\CO$-submodule of $C(A(\Zp),\CO)$. 

\subsection{Review of equivariant measures}\label{patching}
Let $\alpha$ be a character as in Definition \ref{equivmeas}.  
As in in \cite[Lemma 7.4.2]{EHLS}, we can identify an equivariant measure $\phi \in Meas(A,\CV_V)$ of weight $\kappa$
\begin{equation}\label{rho}
\phi(a\cdot f) = \kappa(a) \cdot a\cdot \phi(f), ~~\forall a \in A(\Zp), f \in C(A(\Zp),\CO)
\end{equation}
with a collection 
\begin{equation}\label{phir} 
(\phi_{r,\chi}) \in Hom_{\Lambda_{\CO}(A)}(C_r(A,\CO)\kappa\cdot\chi,\CV_V),
\end{equation}
satisfying a certain distribution relation, written 
$$\eta^*_r(\phi_{r+1, \chi}) = \phi_{r, \chi};$$
we refer the reader to \cite{EHLS} for the definition.   (There is a misprint in \cite{EHLS}:  the factor corresponding to $a\cdot$ after $\alpha(a)$ in \eqref{rho} is missing.)

In the application in this paper, $A$ is the torus $T'_y$.   The classical weights are Zariski dense in $Spec(\Lambda')$.  Recall that we have defined $[\kappa]'$ to be the restriction of the weight $\kappa$ to $T'_y$.  We define $\Lambda'_\kappa$ (which we could also write $\Lambda_{[\kappa]'}$)  to be the quotient of $\Lambda'$ corresponding to the Zariski closure of highest weights of $T'_y$ of the form $\kappa'_b := [\kappa]' + b(1,0,\dots,0)$.  As a ring, $\Lambda'_\kappa$ is isomorphic to $\Lambda'_0$, which is just the Iwasawa algebra $\CO[[T]]$.  We assume our chosen automorphic representation $\tau$ of $U(V')$ is of weight contained in $Spec(\Lambda'_\kappa)$.  In the discussion above, 
$[\kappa]'\cdot \chi$ is taken to be a character $\kappa'_b$, which we henceforth abbreviate $\kappa'$.
In other words, we don't assume that
$\kappa' = [\kappa]'$, but we do want $\kappa'$ to correspond to a classical point of $Spec(\Lambda'_\kappa)$.


We choose the filtration $(A_r) = (T'_{y,r})$
to be compatible with the filtration $K'_{p,r}$ of $K'_p$.
 In \eqref{phir} we can replace
$\Lambda_{\CO}(T'_y)$ by its quotient $\Lambda'_\kappa$, which is an Iwasawa algebra in one variable.  Let
$\CJ \subset \Lambda_{\CO}(T'_y)$ denote the kernel of the homomorphism to $\Lambda'_\kappa$.
Let $(\phi_{r,\chi})$ be as above, and define
$$\phi'_{r,\chi} = res_{V'}\circ \phi_{r,\chi} \in Hom_{\Lambda'_\kappa}(C_r(T'_y,\CO)[\CJ]\kappa\cdot\chi,\CV_{V'}), r \geq 0$$
Here $C_r(T'_y,\CO)[\CJ] \subset C_r(T'_y,\CO)$ is the $\CO$-submodule of functions annihilated by $\CJ$.

As in \cite{EHLS}, we define contraction of $p$-adic measures with Hida families by fixing a sufficiently regular classical weight 
$\kappa' = [\kappa]'\cdot \chi$ of $T'_y$ and taking the limit over $r$ of pairings of the the $\phi'_{r,\chi}$ with the level $K'_r$
components of a fixed Hida family.   
We start with the equivariant measure $\mu_F$ of weight $\kappa$, then, and define  an element $\mu'_F$ of
$Hom_{\Lambda'_\kappa}(C(\ZZ_p^\times,\CO),\CV_{V'})$ by  restricting $\mu_F$ via the map $res_{V'}$;
let $(\phi'_{r,\chi}) = (\phi'_{r,\chi,\tau})$ be the corresponding collection of homomorphisms in
$Hom_{\Lambda'_\kappa}(C_r(A,\CO)[\kappa]'\cdot\chi,\CV_{V}')$.

Let $e'_\tau$ denote the ordinary projector $\CV_{V'} \ra \CV_{V'}^{\ord}$, composed with localization at the maximal
ideal $\mathfrak{m}_\tau$.  
For $\kappa' = [\kappa]' \cdot \chi$ sufficiently regular, it follows from \eqref{control} that the image of $e'_\tau\circ \phi'_{r,\chi}$ is contained in
\begin{equation}\label{du0}
S^{\ord}_{\kappa'}(K'_r,\CO) \isoarrow Hom_{\CO}(\hat{S}^{\ord}_{\kappa'}(K'_r,\CO),\CO).
\end{equation}
More precisely, denoting localization at $\fm_\tau$ by the subscript $_\tau$, the image of $e'_\tau\circ \phi'_{r,\chi}$ lies in 
\begin{equation}\label{du}
S^{\ord}_{\kappa'}(K'_r,\CO)_\tau \isoarrow Hom_{\CO}(\hat{S}^{\ord}_{\kappa'}(K'_r,\CO)_\tau,\CO),
\end{equation}
where the left-hand side is the image of $S^{\ord}_{\kappa'}(K'_r,\CO)$ in the right-hand side of \eqref{du0} after
localization at $\fm_\tau$.

\subsection{Application of the Gorenstein Hypothesis \ref{gor}}  
{  }

The right hand side of \eqref{du} is isomorphic to
$\hat{S}^{\ord}_{\kappa'}(K'_r,\CO)_\tau$ as finite free module of rank $M$ equal to the $\CO$-rank of $\hat{I} = \hat{I}_\tau$ over $\TT_{\tau}$.    
We choose
$\TT_\tau$ bases $\hat{m} \in \hat{M}_\tau$ and $\omega \in \Omega_\tau$, and a primitive $f'_S \in \hat{I}$,
and obtain the following diagram:

\begin{equation}\label{con}
\begin{aligned}
S^{\ord}_{\kp}(K'_r,\CO)_\tau &\isoarrow 
Hom_\CO(\hat{S}^\ord_{\kp}(K'_r,\CO)_\tau),\CO)) \\ &\isoarrow 
Hom_\CO([\hat{I}\otimes_\CO Hom_{\CH_S}(\hat{I},\hat{S}^\ord_{\kp}(K'_r,\CO)_\tau)],\CO)
 \\  &\overset{\hat{m}}{\lra}  Hom_\CO(\hat{I}\otimes_\CO\TT_{r,\kp,\tau} ,\CO) 
 \\ &\overset{\omega}\lra ~Hom_\CO(\hat{I},\CO)\otimes_\CO\TT_{r,\kp,\tau}
 \\ &\overset{f'_S\otimes id}\lra ~~~\TT_{r,\kp,\tau},
\end{aligned}
\end{equation}
where the last line inserts $f'_S$ in the factor $Hom_\CO(\hat{I},\CO)$.
  

We now apply the discussion of \S \ref{patching}  when $(\phi_{r,\chi})$ is attached to the measure $\mu_F$.  As in \cite[\S 7.4]{EHLS}, 
the collection 
$$(\phi'_{r,\chi,\tau} \in S^{\ord}_{\kappa'}(K'_r,\CO)_\tau)_{r \geq 0}$$
 patch together to an element $L^0(F,\tau)$ of a finite free rank $M$ $\TT_\tau$-module.  By choosing $\hat{m}, \omega$, and $f'_S$ as
 above, we identify $L^0(F,\tau)$ with an element
 $$L(F,\tau) := L(F,f'_S,\tau,\hat{m},\omega) \in  \TT_\tau.$$
as in \cite[Proposition 7.4.10]{EHLS}.   Note, however, that this element is not independent of the choices.  In particular, the bases
$\hat{m}$ and $\omega$  of their respective rank $1$ $\TT_\tau$-modules
are only defined up to multiplication by units in $\TT_\tau^\times$.  In general there seems to be no
canonical choice of these bases, in contrast to the familiar case of new forms for $GL(2)$, when the leading term of the $q$-expansion defines a canonical choice.


\subsection{Pairings and periods}\label{Hidafam}


 
 We summarize the construction in the previous section.

\begin{thm}\label{specialization}  Suppose $F$ is classical of weight $\kappa$.  Let $\TT_\tau$ be a component of the ordinary Hecke algebra for cusp forms on $U(V')$; we admit Hypotheses \ref{gor} and \ref{mini}.  Let $L(F,\tau)$ denote the element
$L(F,f'_S,\tau,\hat{m},\omega) \in  \TT_\tau$ constructed above.  Let $x \in Spec(\TT_\tau)$ lie over a weight $\kappa' \in Spec(\Lambda'_\kappa)$, and suppose it corresponds to an eigenform $f'_x = f'_S\otimes  f^{\prime,S}_{\hat{m},x}$ as in \eqref{fpip}.  Then 
\begin{equation}\label{specialx}
\begin{aligned}
&L(F,\tau,x) =  [\theta^{\kappa,\kappa'}(F),f'_x] \\ = & \int_{[U(V')]}  \delta^{\kappa,\kappa'}(F)(g') \cdot f'_x(g') dg' 
= P_{U(V')}(\delta^{\kappa,\kappa'}(F), f'_x). 
\end{aligned}
\end{equation}
\end{thm}
\begin{proof}   Since $f'_x$ is anti-ordinary, this follows immediately from Proposition \ref{tkk}. 
\end{proof}

\begin{remark}  The function $L(F,\tau)$ can be described more canonically as an element of
$\TT_\tau\otimes_\CO \hat{M}_\tau\otimes \Omega_\tau$, or alternatively as a section of the line bundle on
$Spec(\TT_\tau)$ corresponding to the module $\hat{M}_\tau\otimes \Omega_\tau$.    In cases involving elliptic
modular forms (for example, in \cite{HT}) the theory of the $q$-expansion provides a canonical everywhere non-vanishing
section of this line bundle.   I don't know whether or not to expect the corresponding line bundle to be trivial 
more generally, when the Gorenstein condition is satisfied.  The question will be examined more carefully in the sequel
to this paper.
\end{remark}

\section{The Ichino-Ikeda formula and the main theorem}  

The formula is given by the Ichino-Ikeda N. Harris conjecture:
\begin{conj} \label{conjectureII} 
Let $f \in \pi$, $f' \in \pi'$ be factorizable vectors.
Then there is an integer $\beta$, depending on the $L$-packets containing $\pi$ and $\pi'$, such that
$$\CP(f,f') = 2^{\beta}\Delta_{H} \CL^S(\pi,\pi') \ \prod_{v \in S} I^*_v(f_v,f'_v).$$
\end{conj}

Here
$$\CP(f,f'):=\cfrac{|P_{U(V')}(f,f')|^{2}}{<f,f> <f',f'>}.$$
and
\begin{equation}\label{421} \CL^S(\pi,\pi'): = \frac{L^S(\tfrac{1}{2},\Pi\otimes \Pi')}{L^S(1,\Pi,As^{(-1)^n})L^S(1,\Pi',As^{(-1)^{n-1}})} .\end{equation}
where $\Pi = BC(\pi)$, $\Pi' = BC(\pi')$.  Moreover, the local Euler factors $I^*_v$ are given by the explicit formula
\begin{equation}\label{Euler}
I^*_v(f_v,f'_v) = [c_{f_v}(1)c_{f'_v}(1)]^{-1}\cdot \int_{U(V')_v} c_{f_v}(g')c_{f'_v}(g') dg',
\end{equation}
where $c_{f_v}(g) = <\pi(g)f_v,f_v>_{\pi_v}$, for a fixed inner product $<\bullet,\bullet>_{\pi_v}$, and likewise for the matrix coefficient $c_{f'_v}$. Normalizations are explained in the references cited in the following theorem.  

\begin{thm}[\cite{BP18,BPLZZ, X,Z14}]\label{IItheorem}  Suppose $\pi$ and $\pi'$ are everywhere tempered and either (a) there is a non-archimedean place $v$ of $\QQ$ such that $BC(\pi_v)$ and $BC(\pi'_v)$ are
supercuspidal, or (b) both $\pi$ and $\pi'$ are stable in the sense of \cite{BPLZZ}.
  Then Conjecture \ref{conjectureII} holds.
\end{thm}

We apply this when $f = \delta^{\kappa,\kappa'}(F)$ and $f' = f_x$ in \eqref{specialx}.    If $f_x$ is classical we write it in the form
$$f_x = \frac{\overline{f^{hol}_x}}{<f^{hol}_x,f^{hol}_x>}$$
where $f^{hol}_x$ is an  holomorphic modular form (of weight $\kappa'_x$) rational over $\Qbar$, and the denominator is the Petersson inner product.  When $\pi$ and $\pi'$ are in the discrete series at archimedean places it is known thanks to a long list of people, ending with Caraiani, that $\pi$ and $\pi'$ are necessarily tempered everywhere.  
Then we have the following theorem.  
\begin{thm}\label{main}  Suppose $F$ is classical of weight $\kappa$, corresponding to the cuspidal automorphic representation $\pi$.  Fix an anti-ordinary anti-holomorphic automorphic representation $\tau$ of $U(V')$ of weight $\kappa'$, where $\kappa'$ is in $Spec(\Lambda'_\kappa)$.  Suppose $\pi$ and $\pi' = \tau$ satisfy hypotheses (a) or (b) of Theorem \ref{IItheorem}.     Let 
$f = \delta^{\kappa,\kappa'}(F) \in \pi$ be a vector $\otimes'_{v} f_v$, which is unramified outside a finite set $S$ of places containing $p$ and $\infty$.  Let $x \mapsto f'_x$ be a Hida family over $Spec(\Lambda'_\kappa)$.  We assume that, for every classical point $x$, with $f'_x$ corresponding to the automorphic representation $\tau_x$ of $U(V')$,
$f'_x$ is a factorizable vector of the form $\otimes_{v \notin S} f'_{x,v}\otimes f'_S \in \otimes_v\tau_{x,v}$, with $f_{x,v}$ spherical for $v \notin S$,  $f_{x,p}$ the anti-ordinary vector in $\tau_{x,p}$ as in \S 8.3 of \cite{EHLS}.   

We admit Hypotheses \ref{gor} and \ref{mini}
and assume $f'_S = \sum_j \otimes_{v\in S} f'_{v,j} \in \hat{I}$ is a primitive vector as in \S \ref{local}.
Fix generators $\hat{m} \in \hat{M}_\tau$ and $\omega \in \Omega_\tau$ and define $L(F,\tau) = L(F,f'_S,\tau,\hat{m},\omega)$ as in Theorem \ref{specialization}.  
Then for every classical point, the function $L(F,\tau,x)$ is an algebraic number that satisfies
$$|L(F,\tau,x)|^2 = 2^\beta\Delta_H|\delta^{\kappa,\kappa'}(F)|^2<f'_x,f'_x>^2\cdot \CL^S(\pi,\tau_x)\cdot Z_S(x)$$
where
\begin{equation}\label{localZ}
Z_S(x) = Z_\infty(x)\cdot Z_p(x)\cdot \sum_j\prod_{v \in S\setminus \{p,\infty\}} I^*_v(f_v,f'_{v,j}).
\end{equation}

Next, $Z_\infty$ is the Euler factor attached to $\delta^{\kappa,\kappa'}(F_\infty) \in \pi_\infty$ and $f_{x,\infty} \in \tau_{x,\infty}$.
Finally, $Z_p(x)$ is the Euler factor attached to the specified vectors $f_p$ and $f'_{x,p}$.
\end{thm}
\begin{proof}  By the cuspidality hypotheses on $\pi$ and $\tau$ and Theorem \ref{IItheorem} the formula in Conjecture \ref{conjectureII} is valid.  The theorem then follows by combining Theorem \ref{specialization} with \eqref{421} and \eqref{Euler}.
\end{proof}

\begin{remark}   If we don't insist that $f'_S$ be primitive, we can arrange that $Z_S$ be a product and that the local factors
for $v  \in S\setminus \{p,\infty\}$.  are volume terms.  
\end{remark}

\section{Open questions}

\subsection{Local factors at $p$}\label{locp}  The most intriguing open question is the determination of the local factor $I^*_p(f_p,f'_{x,p})$ in Conjecture \ref{conjectureII}.  Specifically, the anti-ordinary vector $f'_{x,p}$ has been identified in \cite{EHLS} as a collection of explicit vectors, of increasing level -- bounded below by a constant determined by the level at $p$ of the component $\tau_x$ of $\tau$, which belongs to the (possibly ramified) principal series of $GL(n-1,\Qp)$.  On the other hand,
$f_p$ must have the property corresponding to $p$-depletion in the classical context of elliptic modular forms.  We can start by replacing $F$ by $\Theta^1(F)$, where $\Theta^1$ is the operator introduced in Remark \ref{idempotent}.  Assuming $\kappa$ sufficiently regular, it follows from Hida's control theorem that $\Theta^1(F)$ is a classical holomorphic modular form, but of level divisible by $p$.  This probably suffices to determine the vector $f_p \in \pi_p$.  

The representation $\pi_p \times \tau_{x,p}$ can be viewed as a representation of $GL(n,\Qp)\times GL(n-1,\Qp)$, and the Ichino-Ikeda local factors $I^*_p(f_p,f'_{x,p})$ can be computed in terms of Jacquet-Piatetski-Shapiro-Shalika local factors for $GL(n)\times GL(n-1)$.  As Beuzart-Plessis explained to me, this was first observed by
Waldspurger; and independently and more explicitly in \S 18.4 of \cite{SV}.   Thus it suffices to compute the Jacquet-Piatetski-Shapiro local factors for our chosen vectors, both of which belong to principal series representations.  Since these factors behave well with respect to parabolic induction, so the calculation may not be as difficult as it appears.  

\subsection{Local factors at $\infty$}\label{locinf} The proof of Conjecture \ref{conjectureII}, in the cases in which is is known, is based on a comparison of the local Euler factors \eqref{Euler} at all places with corresponding local factors in the Jacquet-Piatetski-Shapiro-Shalika integral representation of the Rankin-Selberg $L$-functions for $GL(n)\times GL(n-1)$.  In our situation, $\delta^{\kappa,\kappa'}(F_\infty)$ and $f_{x,\infty}$ are vectors in discrete series representations of $U(r,s)$ and $U(r,s-1)$, respectively; the comparison depends on a transfer of test functions on $U(r,s)\times U(r,s-1)$ to $GL(n,\CC)\times GL(n-1,\CC)$.  As in the previous section, the local 
Euler factors  in the latter situation can be studied by means of parabolic induction, so the computation of local factors at $\infty$ mainly depends on understanding the local transfer.

\subsection{Maximal dimension}  As $b$ varies among positive integers, the $p$-adic modular forms $\Theta^b(F)$ can be paired not only with classical forms of weight $[\kappa]'+ (b,0,\dots)$ but with those of more 
general weights in the decomposition \eqref{reskappa}.  The function $L(f,\tau)$ should thus extend to a function on $min(r,s)$-dimensional Hida families.  This will be considered in future work with Ellen Eischen.

\subsection{Extension to coherent cohomology in higher degree}    Conjecture \ref{conjectureII} applies to many central values of motivic $L$-functions that are not realized as pairings of holomorphic and anti-holomorphic modular forms.  In many cases they can nevertheless be realized as cup products in higher coherent cohomology; some examples are considered in \cite{GHLin}.   Pilloni's recent development of {\it higher Hida theory} shows that higher coherent cohomology classes can also be deformed in ordinary families.  Work in progress by Loeffler, Pilloni, Skinner, and Zerbes aims to use this theory to study the $p$-adic behavior of special values of $L$-functions of certain automorphic representations, for groups of low dimension, that are known to be related to cup products in coherent cohomology.  In future work with Eischen and Pilloni, we hope to make a systematic study of square-root $p$-adic $L$-functions for $U(n)\times U(n-1)$, whenever coherent cohomology can be applied.

Many of the period integrals in Conjecture \ref{conjectureII} involve coherent cohomological representations but are not identified as cup products.  There should  be square root $p$-adic $L$-functions in these cases as well, but it is not at all clear how they can be defined.

\subsection{Slopes and the Panchishkin condition}  General conjectures on $p$-adic $L$-functions predict that they can be constructed for quite general motives, but that 
they belong to Iwasawa algebras, or finite extensions thereof, only when the motive satisfies a {\it Panchishkin condition}, which is the  analogue for the $p$-adic slope filtration of the condition on the Hodge filtration that guarantees that a special value is {\it critical} in Deligne's sense.  No such conjectural restriction has been formulated, as far as I know, for the existence of $p$-adic analytic functions with values in finite extensions of Hida's ordinary Hecke algebra (itself a finite extension of a multivariable Iwasawa algebra) that interpolate square roots of normalized critical values of $L$-functions.  The method of the present paper presupposes that the forms on $U(V')$ vary in an ordinary family but impose no restriction on the forms in the larger group $U(V)$.  Is the construction here consistent with general conjectures?

\subsection*{}


\begin{appendices}

\section{Appendix}\label{appen}  
\subsection{Review of Shimura data for unitary groups}\label{ShU}  The Shimura datum $(U(V),Y_V)$ is introduced in \S \ref{sec2}, following the discussion in \cite[\S 27]{GGP}.   We review the definition given there, in the simpler case treated here where $\CK$ is an imaginary quadratic field.
Let $V$ and $(r,s)$ be as in  the beginning of \S \ref{sec2}.  Let $GU(V)$ denote the algebraic group over $\QQ$ of unitary similitudes of $V$; in other words, it is the subgroup of $R_{\CK/\QQ}(GL(V))$ of automorphisms of $V$ that preserve the hermitian form up to a scalar.   Let $\Ss$ denote the Serre torus $R_{\CC/\RR}\mathbb{G}_{m,\CC}$.    Define a map
$$h_V = \Ss \ra GU(V)(\RR)$$
by the formula
\begin{equation} \label{hV} h_{V}(z) = \begin{pmatrix}  zI_{r} & 0 \\ 0 & \bar{z}I_{s} \end{pmatrix},
\end{equation}
where $I_r$ and $I_s$ are the identity matrices of size $r$ and $s$, respectively.  Denote by $X_V$  the $GU(V)(\RR)$-conjugacy class 
of homomorphisms  $\Ss \ra GU(V)_\RR$ containing $h_V$; thus  $(GU(V),X_V)$ is a Shimura datum

Let $V_1$ denote the vector space $\CK$, endowed with a hermitian form of signature $(0,1)$ at the chosen complex embedding $\iota$,
 and define $h_{V_1} = h_\Sigma:  \Ss \ra GU(V_1)(\RR)$ by analogy with \eqref{hV}; as above, we thus have a Shimura datum 
 $(GU(1),h_{V_1})$   Let
$G'_V \subset GU(1)\times GU(V)$ be the subgroup consisting of $(t,g)$ with $\nu(t) = \nu(g)$.  
Then the map
$$h'_V : \Ss \ra [GU(V_1) \times GU(V)](\RR);  h'_V(z) = (h_{V_1}(z),h_V(z))$$
has image contained in $G'_V(\RR)$;  thus we may write $h'_V:  \Ss \ra G'_V(\RR)$.  
Let $X'_V$ denote the $G'_V(\RR)$-conjugacy class of homomorphisms $h: \Ss \ra G'_V(\RR)$ containing $h'_V$.  Then $(G'_V,X'_V)$ is a Shimura datum, and the inclusion map
$G'_V \hookrightarrow GU(V) \times GU(1)$ induces a morphism of Shimura data
$$(G'_V,X'_V) \rightarrow (GU(V)\times GU(1),X_V \times X_1).$$

There is a natural map 
\begin{equation}\label{GUtoU}  u:  G'(V) \ra U;  u(t,g) = t^{-1}g, \forall (t,g) \in G'(V) \subset GU(1)\times GU(V)
\end{equation}
The map taking $h \in X'_V$ to $u\circ h:  \Ss \ra U(V)(\RR)$ then defines a map of Shimura data
$$(G'_V,X'_V) \rightarrow (U(V),Y_V)$$
where $Y_V$ is the $U(V)(\RR)$-conjugacy class defined by this map.   Unlike the other Shimura data introduced in this section, the pair
$(U(V),Y_V)$ is not of PEL type.  It is of abelian type, however, and arithmetic of the Shimura variety $S(U(V),Y_V)$ has been studied, nevertheless, in \cite{KSZ}.  In this paper we write $S(V)$ instead of $S(U(V),Y_V)$, and if $K \subset U(V)(\af)$ is an open compact subgroup,the corresponding finite level Shimura variety is denoted $_KS(V)$.

\subsection{Review of anti-holomorphic and anti-ordinary forms}\label{antiord}    The complex structure on the hermitian symmetric space $Y_V$ determines the complex structure on the Shimura variety $Sh(V)$, and one thus has a well-defined notion of holomorphic sections of the canonical extensions of automorphic vector bundles on toroidal compactifications of $Sh(V)$.  Such a holomorphic section $\phi$ gives rise by a {\it canonical trivialization} to an automorphic form  $F_\phi$ on $G(V)(\QQ)\backslash G(V)(\ad)$, and thus to an automorphic representation $\pi = \pi(\phi)$ of $G(V)(\ad)$.  Such a $\pi$ is called {\it holomorphic}, and $F_\phi$ is a highest $K_\infty$-type vector for an
appropriately chosen maximal compact subgroup $K_\infty \subset G(V)(\RR)$.  For all this, see \cite[\S\S 2.4, 2.5]{H97}, especially (2.4.4 bis).   In particular, if $\pi$ is holomorphic then its archimedean component $\pi_\infty$ is isomorphic to a holomorphic discrete series representation $\DD_\kappa$, as defined in \S \ref{holomo}. 

Then an anti-holomorphic automorphic representation is just the complex conjugate $\bar{\pi}$ of a holomorphic automorphic representation; equivalently, $\pi'$ is anti-holomorphic if $\pi'_\infty$ is isomorphic to the contragredient of a holomorphic discrete series representation
$\DD_\kappa$.    As functions on $G(V)(\QQ)\backslash G(V)(\ad)$, anti-holomorphic automorphic forms are just the complex conjugates of holomorphic automorphic forms.  

We use the term {\it anti-ordinary form} as in \cite{EHLS} to denote an element of the dual module $\hat S_{\kappa}^\ord(K;\CO)$ defined
in \eqref{hatS}, for appropriate level $K$ and weight $\kappa$.  The property of being anti-ordinary is determined by the valuations of a family of local Hecke operators at primes dividing $p$ that are {\it dual} to the $U$-operators used to define Hida's ordinary subspace.  The details can be found in \cite{EHLS}, \S 6.6.6 and \S 8.3.5.  For our purposes here, the main properties of anti-ordinary forms are the following:

\begin{itemize}
\item[(i)] $\hat{S}_{\kappa}^\ord(K;\CO)_\tau$ is identified with an $\CO$-lattice in 
$$
\oplus_{\tau'\in\mathcal{S}(K,\kappa,\tau)} \tau_{p,r}^{\prime,\flat,\aord}\otimes\tau_S^{\prime,\flat,K_S},
$$
\item[(ii)]   The pairing of $S_\kappa(K,\CO)$ with $\hat{S}^{ord}(K,\CO)$ factors through the ordinary projection
$$S_{\kappa}(K,\CO) \ra S_\kappa(K,\CO)^{ord}.$$
\end{itemize}

Here $\tau_{p,r}^{\prime,\flat,\aord}$ is an explicit one-dimensional anti-ordinary eigenspace (at level $K_{p,r}$, see \cite{EHLS} for details) of the local component $\tau^\flat_p$ of $\tau^\flat$; the anti-ordinary eigenspace is defined explicitly in the model of $\tau_p^\flat$ as a principal series representation.  The set $\mathcal{S}(K,\kappa,\tau)$ is roughly the set of automorphic representations
$\tau'$ that are congruent modulo $p$ to $\tau$.  
Property (i) follows from Lemma 6.6.12 (ii) and the discussion preceding Lemma 6.6.11 of \cite{EHLS}, and (ii) follows from  \cite[Lemma 8.3.4 (iii)]{EHLS}.

\end{appendices}

\end{document}